\documentclass[11pt,reqno]{amsart}
\usepackage{amssymb}
\usepackage{amsmath}
\usepackage{amsthm}
\usepackage{graphicx}
\usepackage{caption}
\usepackage{bm}
\usepackage[colorlinks=true,urlcolor=blue,
citecolor=red,linkcolor=blue,linktocpage,pdfpagelabels,
bookmarksnumbered,bookmarksopen]{hyperref}
\usepackage[titletoc,title]{appendix}
\numberwithin{equation}{section} 
\newcounter{cont}[section] 

\newtheorem{thm}[cont]{Theorem}
\newtheorem{prop}[cont]{Proposition}
\newtheorem{lem}[cont]{Lemma}

\theoremstyle{definition}

 \theoremstyle{remark}
 \newtheorem{rem}[cont]{Remark}

\newcommand{\R}{\mathbb{R}}
\newcommand{\e}{\varepsilon}
\newcommand{\ut}{\tilde{u}}
\newcommand{\ub}{\bar{u}}

\begin{document}
\baselineskip=16pt

\title[Slow dynamics for hyperbolic Cahn--Hilliard equation]{Slow dynamics for the hyperbolic Cahn--Hilliard equation in one space dimension}

\author[R. Folino]{Raffaele Folino}
\address[Raffaele Folino]{Dipartimento di Ingegneria e Scienze dell'Informazione e Matematica, Universit\`a degli Studi dell'Aquila (Italy)}
\email{raffaele.folino@univaq.it}
	
\author[C. Lattanzio]{Corrado Lattanzio}
\address[Corrado Lattanzio]{Dipartimento di Ingegneria e Scienze dell'Informazione e Matematica, Universit\`a degli Studi dell'Aquila (Italy)}
\email{corrado@univaq.it} 
 
\author[C. Mascia]{Corrado Mascia}
\address[Corrado Mascia]{Dipartimento di Matematica,
Sapienza Universit\`a di Roma (Italy)}
\email{mascia@mat.uniroma1.it}


\maketitle

\begin{abstract} 
The aim of this paper is to study the metastable properties of the solutions to a  hyperbolic relaxation of the classic Cahn--Hilliard equation in one space dimension,
subject to either Neumann or Dirichlet boundary conditions. 
To perform this goal, we make use of an ``energy approach", already proposed for various evolution PDEs, including the Allen--Cahn and the Cahn--Hilliard equations.
In particular, we shall prove that certain solutions  maintain  a {\it $N$-transition layer structure}  for a very long time, thus proving their metastable dynamics.
More precisely, we will show that, for an exponentially long time, such solutions are very close to piecewise constant functions assuming only the minimal points of the potential, 
with a finitely number of transition layers, which move with an exponentially small velocity. 
\end{abstract}

\section{Introduction}
\subsection{Hyperbolic relaxation of Cahn--Hilliard equation}
In this paper, we are interested in the metastable dynamics of the solutions to the \emph{hyperbolic Cahn--Hilliard equation}
\begin{equation}\label{CahnHill+ine}
	\tau u_{tt}+u_t=(-\e^2u_{xx}+W'(u))_{xx},
\end{equation}
where the function $W:\R\rightarrow\R$ is a double well potential with wells of equal depth;
precisely, we assume that $W\in C^3(\R)$ and 
\begin{equation}\label{eq:ass-W}
	W(\pm1)=W'(\pm1)=0, \qquad W''(\pm1)>0 \qquad \mbox{ and } \qquad W(u)>0, \quad \mbox{ for } \, u\neq\pm1. 
\end{equation}
Roughly speaking, $W$ is a smooth nonnegative function with global minimum $0$ reached only at two points, $\pm1$.
The simplest example is $W(u)=\frac14(u^2-1)^2$.

Equation \eqref{CahnHill+ine} is a \emph{hyperbolic variation} of the classic Cahn--Hilliard equation 
\begin{equation}\label{CahnHill}
	 u_t=(-\e^2u_{xx}+W'(u))_{xx},
\end{equation}
originally proposed in \cite{Cahn-Hilliard} to model phase separation in a binary system at a fixed temperature,
with constant total density and where $u$ stands for the concentration of one of the two components.
The hyperbolic version \eqref{CahnHill+ine} has been firstly proposed by Galenko in \cite{Galenko}, 
following the classical Maxwell--Cattaneo modification of the Fick's diffusion law \cite{Cat,JP89a,JP89b}; 
for further physical details about this model we also refer to \cite{GalenkoJou,LeZaGa}. 
In order to briefly describe such derivation, we recall the one-dimensional continuity equation for the concentration $u$:
\begin{equation}\label{eq:cont}
		u_t+J_x=0,
\end{equation}
where $J$ is its flux. 
In the case of the Cahn--Hilliard equation \eqref{CahnHill} the flux $J$ is related to the concentration $u$ according to the law
\begin{equation}\label{eq:flux}
	J=-(-\e^2u_{xx}+W'(u))_x.
\end{equation}
By substituting \eqref{eq:flux} in \eqref{eq:cont}, we obtain \eqref{CahnHill}.
On the other hand, in the case of the hyperbolic modification \eqref{CahnHill+ine}, 
the equilibrium between flux and derivative of the quantity $\e^2u_{xx}-W'(u)$ is asymptotical with a time-scale measured by the relaxation parameter $\tau>0$,
namely 
\begin{equation}\label{eq:flux-rel}
	\tau J_t+J=-(-\e^2u_{xx}+W'(u))_x.
\end{equation}
Multiplying by $\tau$ and differentiating equation \eqref{eq:cont} with respect to $t$, and differentiating equation \eqref{eq:flux-rel} with respect to $x$, 
we obtain the hyperbolic Cahn--Hilliard equation \eqref{CahnHill+ine}.
Notice that, taking (formally) the limit as $\tau\to0^+$, one obtains \eqref{eq:flux} from \eqref{eq:flux-rel} 
and the classic Cahn--Hilliard equation \eqref{CahnHill} from \eqref{CahnHill+ine}.

We consider equation \eqref{CahnHill+ine} in a bounded interval of the real line $I=[a,b]$, with initial conditions
\begin{equation}\label{initial}
	u(x,0)=u_0(x), \quad u_t(x,0)=u_1(x), \qquad x\in[a,b],
\end{equation}
and either homogeneous Neumann boundary conditions
\begin{equation}\label{boundary-Neu}
	u_x(a,t)=u_x(b,t)=u_{xxx}(a,t)=u_{xxx}(b,t)=0, \qquad \forall\,t\geq0,
\end{equation}
or Dirichlet boundary conditions 
\begin{equation}\label{boundary-Dir}
	u(a,t)=\pm1, \quad u(b,t)=\pm1 \qquad \mbox{ and } \qquad u_{xx}(a,t)=u_{xx}(b,t)=0, \qquad \forall\,t\geq0.
\end{equation}
It is worth to observe that, referring to the original Cahn--Hilliard equation 
\eqref{CahnHill}, homogeneous Neumann boundary conditions \eqref{boundary-Neu} are physically more relevant,
because they guarantee mass conservation and are equivalent to $u_x(a,t)=u_x(b,t)=0$, plus no flux condition at the boundary, $J(a,t)=J(b,t)=0$. 
For the hyperbolic model \eqref{CahnHill+ine} these properties are still valid, provided extra assumptions for initial data; 
in particular, concerning mass conservation, see Section \ref{sec:Neu} below. 

The problem of existence and uniqueness of solutions to the equation \eqref{CahnHill+ine} in a bounded interval of the real line
with different boundary conditions has been widely studied.
In particular, in the case of homogeneous Neumann boundary conditions \eqref{boundary-Neu}, 
we recall the results of \cite{Debussche}.
Let us introduce the unbounded linear operator
\begin{equation*}
	\mathcal A:=\frac{d^4}{dx^4}, \qquad D(\mathcal A):=\{u\in H^4(a,b) : u'(a)=u'(b)=u'''(a)=u'''(b)=0\}.
\end{equation*}
The operator $\mathcal A$ is positive, self-adjoint and possesses an orthonormal basis of eigenfunctions  $\{\omega_j\}_{j\in\mathbb{N}}$.
Denote by $\{\lambda_j\}_{j\in\mathbb{N}}$ the eigenvalues of $\mathcal A$ and recall that for any $s\in\R$,
the operator $\mathcal A^s$ is defined by 
\begin{equation*}
	\mathcal A^su:=\sum_{j=1}^{+\infty}\lambda^s_j u_j\omega_j, \qquad \mbox{ for } \, u:=\sum_{j=1}^{+\infty}u_j\omega_j.
\end{equation*}
Denoting by $V_s:=D(\mathcal A^{s/4})$, we have 
\begin{equation*}
	u\in V_s \qquad \mbox{ if } \quad \sum_{j=1}^{+\infty}\lambda_j^{s/2}u_j^2<+\infty.
\end{equation*}
We can say that for $s=1,2,3$ and for any $(u_0,u_1)\in V_s\times V_{s-2}$
the hyperbolic Cahn--Hilliard equation \eqref{CahnHill+ine} with initial data \eqref{initial} and boundary conditions \eqref{boundary-Neu}
possesses a unique solution $(u,u_t)\in C\left([0,T], V_s\times V_{s-2}\right)$, for any $T>0$; see \cite[Theorem 1.2]{Debussche}.
Similar results also hold true in the case of either periodic or homogeneous Dirichlet boundary conditions; see
\cite{Debussche,Gattietal05,ZhengMilani2005} and references therein.
However, the main aim of our investigation is related to metastable properties of solutions, thus we are not interested here in the complete discussion of 
the well--posedness of the aforementioned IBVPs for \eqref{CahnHill+ine}, and therefore in the sequel we shall refer to 
 solutions verifying for any $T>0$
\begin{equation*}
	(u,u_t)\in C\left([0,T], H^3(a,b)\times H^1(a,b)\right),
\end{equation*}
emanating from (sufficiently smooth) initial data.
Notice that, from the discussion above,  in the case of homogeneous Neumann boundary conditions \eqref{boundary-Neu}, 
this is guaranteed if $(u_0,u_1)\in D(\mathcal{A})\times V_1$ 
(because $D(\mathcal{A})\subset V_3$ and when $s$ is positive $V_s$ is a subspace of $H^s(a,b)$). 

Many papers have been devoted to the study of the long-time behavior of the solutions of \eqref{CahnHill+ine} 
and to the comparison of the asymptotic behavior of the solutions of \eqref{CahnHill+ine} with that of \eqref{CahnHill}.
Without claiming to be complete, we recall the following contributions. 
In \cite{Debussche}, the author considers the equation \eqref{CahnHill+ine} with either periodic
or homogeneous Neumann boundary conditions \eqref{boundary-Neu} and initial data \eqref{initial}.
Since it is impossible to find bounded absorbing sets in the whole space $V_s\times V_{s-2}$, $s=1,2,3$, (see the definitions above),
equation \eqref{CahnHill+ine} is restricted to a subspace of $V_s\times V_{s-2}$ and existence of bounded absorbing sets for the restriction is proved.
Moreover, the author shows that there exists a maximal attractor which is compact and connected with respect to the weak topology;
this attractor is close to the one of the Cahn--Hilliard equation \eqref{CahnHill} when $\tau$ is small, and 
the distance between the attractors converges to zero as $\tau\to0^+$.

These results are improved in \cite{ZhengMilani2004,ZhengMilani2005}, where the authors consider equation \eqref{CahnHill+ine} and its viscous version,
that is \eqref{CahnHill+ine} with the additional viscosity term $-\delta u_{txx}$, with homogeneous Dirichlet boundary conditions.
In \cite{ZhengMilani2005} the attention is focused on the existence of global attractors (with respect the strong topology), 
at least when $\tau$ is small, and on the upper semicontinuity of the attractors at $\tau=0$.
The goal of \cite{ZhengMilani2004} is the construction of exponential attractors (also known as inertial sets) and inertial manifolds for $\tau$ and $\delta$ fixed.
In \cite{Gattietal05} a detailed analysis of the longterm properties of the solutions in the singular limit $\tau\to0^+$ is provided
and the authors prove the existence of a regular global attractor for any $\tau\geq0$.
Besides, they construct a continuous (with respect to the limit $\tau\to0^+$) family of regular exponential attractors,
whose common domains of attraction coincide with the entire phase-space.
However, in \cite{Gattietal05} the convergence of the exponential attractors is obtained with respect to a norm that depends on $\tau$.
In \cite{BonGraMir}, the authors construct a family of exponential attractors which is continuous with respect to $\tau$ 
by using a metric that does not depend on $\tau$ as $\tau$ goes to zero.
The aforementioned bibliography is confined to one-dimensional model;
there is vast literature of works about equation \eqref{CahnHill+ine} in more that one spatial dimension
(among others, see \cite{Gattietal2,Grass-Pierr,GraSchZel,GraSchZel2,Segatti}).

In this paper, we want to study the metastable dynamics of the solutions of the hyperbolic Cahn--Hilliard equation \eqref{CahnHill+ine}
with both Neumann \eqref{boundary-Neu} and Dirichlet \eqref{boundary-Dir} boundary conditions.
Before presenting our results, we recall some well-known results on metastability 
for the Allen--Cahn equation and its hyperbolic variants, and for the classic Cahn--Hilliard equation \eqref{CahnHill}. 

\subsection{Metastability for evolution PDEs}
\emph{Metastability} is a broad term describing the persistence of unsteady structures for a very long time;
\emph{metastable dynamics} is characterized by evolution so slow that (non-stationary) solutions \emph{appear} to be stable.
In other words, the phenomenon of metastability appears when the time dependent solution of an evolution PDE, in a first phase,
evolves very slowly in time and after a very long time undergoes a drastic change.
We refer to this type of solutions as metastable states for the evolution PDE. 
The phenomenon of metastability was firstly observed in \cite{Bron-Kohn,Carr-Pego,Carr-Pego2,Fusco-Hale} for the \emph{Allen--Cahn equation} 
\begin{equation}\label{allencahn}
	u_t=\varepsilon^2 u_{xx}-W'(u),
\end{equation}
where $W$ is a potential satisfying \eqref{eq:ass-W}.
When equation \eqref{allencahn} is considered in a bounded interval of the real line with homogeneous Neumann boundary conditions, 
the solution converges to the constant solution $-1$ or $+1$ (the only stable stationary solutions), but in some cases the convergence is exceedingly slow.
If the solution at the time $t=0$ has $N$ transitions between $-1$ and $+1$ with an \emph{efficient} layered structure,
and the diffusion coefficient $\varepsilon$ is small, then the solution maintains these $N$ transitions for an exponentially long time,
i.e. for a time proportional to $e^{C/\varepsilon}$, for some $C>0$ (independent of $\e$).
A complete description of the dynamics of equation \eqref{allencahn} is performed in \cite{Chen}, 
where the author studies generation, persistence and annihilation of the metastable states.
If the initial datum $u_0$ is continuous, independent of $\varepsilon$, and changes sign finitely many times,
then there are four stage in the dynamics.
In the first stage, called \emph{phase separation}, the diffusion term $\varepsilon^2 u_{xx}$ can be ignored 
since $\varepsilon$ is small and $u_0$ is independent of $\varepsilon$;
thus, after this stage, the solution is well approximated by $+1$ or $-1$ except near a finite number of transition points $h_i$. 
In the second stage, \emph{generation of metastable patterns}, diffusion and reaction term compete and, after certain amount of time, balance each other.
After the second stage, we have the persistence of the metastable state for an exponentially long time. 
This stage lasts until the transition points are well-separated,
when they are close enough we have the annihilation of the interfaces, and after every annihilation, 
new metastable patterns with less transitions are developed. 
Finally, all the transitions are eliminated and the solution approaches a constant stable equilibrium.

The pioneering works \cite{Bron-Kohn,Carr-Pego,Carr-Pego2,Fusco-Hale} are devoted to the study in detail of the third stage,
the persistence of a metastable state.
This stage is studied using two different approaches. 
In \cite{Carr-Pego,Carr-Pego2,Fusco-Hale}, the authors use a dynamical approach and 
construct an ``approximately invariant''  $N$-dimensional manifold $\mathcal{M}$, consisting of functions which approximate metastable states with $N$ transition layers. 
If the initial datum is in a small neighborhood of $\mathcal{M}$, then the solution remains near $\mathcal{M}$ for a time proportional to $e^{C/\varepsilon}$. 
Therefore, for an exponentially long time the solution is a function with $N$ transitions between $-1$ and $+1$, 
and this transition structure persists until two transition points are close enough.
In addition, it is possible to determine a system of ordinary differential equations describing the motion of the transition layers,
and prove that the transition points move with exponentially small velocity.

On the other hand, in \cite{Bron-Kohn} the authors propose an energy approach,
based on the fact that the Allen--Cahn equation \eqref{allencahn} is a gradient flow in $L^2(a,b)$
for the Lyapunov functional
\begin{equation}\label{eq:F_eps}
	F_\e[u]=\int_a^b\left[\frac{\e^2}2u_x^2+W(u)\right]\,dx.
\end{equation}
More precisely, if $u^\e$ is a solution of the Allen--Cahn equation \eqref{allencahn} with either Neumann or Dirichlet boundary conditions,
then
\begin{equation}\label{eq:energy-var}
	F_\e[u^\e](0)-F_\e[u^\e](T)=\int_0^T\|u^\e_t(\cdot,t)\|^2_{{}_{L^2}}\,dt, \qquad \quad \forall \,T>0.
\end{equation}
The $\Gamma$-convergence properties of the functional $F_\e$ have been extensively studied 
(among others, see \cite{Modica,Sternberg}).
In particular, if we consider the normalized energy functional 
\begin{equation}\label{eq:P_eps}
	P_\e[u]:=\int_a^b\left[\frac{\e}2u_x^2+\frac{W(u)}{\e}\right]\,dx,
\end{equation}
then in one space dimension the minimum energy to have a transition between $-1$ and $+1$ is asymptotically
\begin{equation}\label{c_0}
	c_0:=\int_{-1}^{1}\sqrt{2W(s)}\,ds.
\end{equation}
The constant $c_0$ is positive and independent of $\e$, and so the normalized energy \eqref{eq:P_eps} is positive and finite as $\e\to0^+$.
Precisely, the results of \cite{Modica,Sternberg} assert that if the sequence $\{u^\e\}$ converges in $L^1(a,b)$ 
to a piecewise continuous function $v$ which has $N$ transitions between $-1$ and $+1$, then
\begin{equation}\label{eq:liminf}
	\liminf_{\e\to0} P_\e[u^\e]\geq Nc_0,
\end{equation}
with equality if the sequence $\{u^\e\}$ is properly chosen, as for  $\{u_0^\varepsilon(x)\}$ constructed  in Section \ref{sec:N-translayer}.
An improvement of \eqref{eq:liminf} permits to study the metastable dynamics of the solutions to \eqref{allencahn}.
Indeed, in \cite[Proposition 2.1]{Bron-Kohn} it has been proved that if $u\in H^1(a,b)$ is close enough in $L^1$ 
to a piecewise continuous function $v$ with $N$ transitions and satisfies $P_\e[u]\leq Nc_0+\e^k$, for a positive integer $k$, then
\begin{equation}\label{lower-BK}
	P_\e[u]\geq Nc_0-C\e^k.
\end{equation}
Using this lower bound on the energy, in \cite{Bron-Kohn} it is proved that the solutions of \eqref{allencahn} 
with either Neumann or Dirichlet boundary conditions and appropriate initial data maintain a \emph{transition layer structure} for a time $T_\e\geq \e^{-k}$ 
and the transition points move slower than $\e^k$.
Hence, in \cite{Bron-Kohn} the authors obtain \emph{algebraically slow motion}.

The lower bound \eqref{lower-BK} can be improved and extended to the case of vector-valued functions (cfr. \cite{Grant,Folino2}).
In particular, we have that if $u$ is close enough to $v$ in $L^1$, then
\begin{equation}\label{eq:lower-G}
	P_\e[u]\geq Nc_0-C\exp(-C/\e).
\end{equation}
Using the latter lower bound, one can prove persistence of metastable states for an exponentially long time.

The results on metastability for the Allen--Cahn equation \eqref{allencahn} can be extended to the \emph{hyperbolic Allen--Cahn equation} 
\begin{equation}\label{hyp-allencahn}
	\tau u_{tt}+g(u)u_t=\e^2u_{xx}-W'(u),
\end{equation}
where $g:\R\to\R$ is strictly positive (see \cite{Folino,FLM,Folino2}).
In particular, in \cite{FLM}, using the dynamical approach, the authors prove existence and persistence of metastable states
for an exponentially long time. 
On the other hand, the energy approach is applied to \eqref{hyp-allencahn} in \cite{Folino}.
In this case, the energy functional is
\begin{equation*}
	E_\e[u,w]=\frac{\tau}{2\e}\|w\|^2_{{}_{L^2}}+P_\e[u],
\end{equation*}
where $P_\e$ is defined in \eqref{eq:P_eps}.
If $u$ is a solution of \eqref{hyp-allencahn} with homogeneous Neumann boundary conditions, then
$E_\e[u,u_t]$ is a non-increasing function of $t$, namely
\begin{equation}\label{eq:energy-var-hyp}
	E_\e[u^\e,u^\e_t](0)-E_\e[u^\e,u^\e_t](T)=\e^{-1}\int_0^T\|u^\e_t(\cdot,t)\|^2_{{}_{L^2}}\,dt, \qquad \quad \forall \,T>0.
\end{equation}
In order to have slow motion of the solutions, the initial velocity $u_1=u_t(\cdot,0)$ has to be very small.
Precisely, in \cite{Folino} it is proved that if the initial profile $u_0$ has a \emph{transition layer structure of order $k$}
and the $L^2$--norm of the $u_1$ is bounded by $C\e^{\frac{k+1}2}$, then 
the solution maintains the same transition layer structure of its initial datum for a time $T_\e\geq \e^{-k}$.
These results are extended to the case of systems (when $u$ is a vector-valued function) in \cite{Folino2},
where, imposing stronger conditions on the initial data, persistence of metastable states for an exponentially long time is obtained.

The phenomenon of metastability is also present in the case of the Cahn--Hilliard equation.
Consider equation \eqref{CahnHill} in a bounded interval of the real line with homogeneous Neumann boundary conditions \eqref{boundary-Neu}; 
a straightforward computation shows that the total mass $\displaystyle\int_a^b u(x,t)\,dx$ is conserved.
Moreover, it is well-known that the solution converges as $t\to+\infty$ to a stationary solution, 
which minimizes the energy $F_\e$ \eqref{eq:F_eps} (cfr. \cite{Zheng1986}), 
and that the minimizers of $F_\e$ subject to the constraint $\displaystyle\int_a^b u\,dx=M$ are monotone functions (cfr. \cite{CarrGurtSlem}).
Therefore, the solutions converge to a limit which has at most a single transition.
However, as in the Allen--Cahn equation, if the initial profile has an \emph{efficient} $N$--layer structure, then the solution maintains
these $N$ transitions for a very long time.
Hence, we have an example of metastable dynamics.
The metastable properties of the solutions to \eqref{CahnHill}-\eqref{boundary-Neu} have been investigated in \cite{AlikBateFusc91,Bates-Xun1,Bates-Xun2} 
by using the dynamical approach, and the authors prove persistence of metastable states for an exponentially long time.
The energy approach performed in \cite{Bron-Hilh} permits to handle both Neumann \eqref{boundary-Neu} and Dirichlet \eqref{boundary-Dir} boundary conditions.
The energy functional $P_\e$ is the same of the Allen--Cahn equation (see \eqref{eq:P_eps}),
but in this case the variation of the energy is not related to the $L^2$--norm as in \eqref{eq:energy-var},
and we have to work with different spaces depending on the boundary conditions.
In \cite{Fife}, the Cahn--Hilliard equation \eqref{CahnHill} with \eqref{boundary-Neu} is derived as the gradient flow in the zero--mean subspace of the dual of $H^1(a,b)$,
while in the case of \eqref{boundary-Dir} the variation of the energy is related to the $H^{-1}$--norm of $u_t^\e$ (for details see \cite{Bron-Hilh}).
By using this approach and the lower bound \eqref{lower-BK}, in \cite{Bron-Hilh} \emph{algebraically slow motion} for the solutions of the Cahn--Hilliard equation is obtained.
These results are improved and extended to the case of Cahn--Morral systems (the vector-valued version of \eqref{CahnHill}) in \cite{Grant}.
 
The aim of this paper is to study the slow evolution of the solutions to the hyperbolic version \eqref{CahnHill+ine} by using the energy approach.
We use the same energy functional $E_\e[u,w]$ of the hyperbolic Allen--Cahn equation.
As we will see in Section \ref{sec:slow}, there is a difference between the boundary conditions \eqref{boundary-Neu} and \eqref{boundary-Dir}. 
Indeed, in order to ensure the energy decreases along solutions of \eqref{CahnHill+ine} with Neumann boundary conditions \eqref{boundary-Neu} 
(Lemma~\ref{lem:enNeu}), we still need mass conservation for the hyperbolic equation, which requires 
 that the initial velocity $u_1$ is of zero mean; see Section \ref{sec:Neu}.
After we know the energy $E_\e[u,w]$ is monotone decreasing in time, we need a lower bound for it that holds true for a larger class of functions: 
we will show that, if $u\in H^1(a,b)$ is such that a primitive of $u$ is close enough in $L^1$
to a primitive of a piecewise continuous function with $N$ jumps between $-1$ and $+1$, 
then the functional $P_\e[u]$ satisfies \eqref{eq:lower-G}; for details, see Proposition \ref{prop:lower} below. 
We will use these results to prove slow evolution of the solutions with boundary conditions \eqref{boundary-Neu} in Section \ref{sec:Neu} 
or \eqref{boundary-Dir} in Section \ref{sec:Dir} by showing that the solutions maintain $N$ transitions 
for a very long time $T_\e$ diverging to $+\infty$ as $\e\to 0^+$.
In addition, in Section \ref{sec:layer}, we shall study the dynamics of such transition points, 
proving in particular that their velocity is very small for $\e\ll 1$, being $O(T_\e^{-1})$.
Finally, Section \ref{sec:N-translayer} is devoted to the construction of an example of initial profile satisfying the needed assumptions.
We conclude the paper by presenting in Section \ref{sec:sys} the results in the case of systems, 
namely when $u$ is vector-valued and $W$ vanishes only in a finite number of points. 
 
\section{Slow motion}\label{sec:slow}
In this section we study the limiting behavior of the solutions to the hyperbolic Cahn--Hilliard equation \eqref{CahnHill+ine} as $\e\to0^+$ 
when the initial profile $u_0$ has a $N$--transition layer structure. 
In particular, using the energy approach we prove existence of metastable states for equation \eqref{CahnHill+ine} and 
show that the solution maintains the same $N$--transition layer structure of the initial profile $u_0$ for a very long time. 
The key tool to use the energy approach to study the slow evolution of the solutions to \eqref{CahnHill+ine} is the energy functional
\begin{equation}\label{energy}
	E_\varepsilon[u,w]:=\frac{\tau}{2\e}\|w\|^2_{{}_{L^2}}+P_\e[u], \qquad \quad \mbox{ where } \quad P_\e[u]:=\int_a^b\left[\frac{\e}{2}u_x^2+\frac{W(u)}{\e}\right]\,dx.
\end{equation} 
First of all, we consider the case of Dirichlet boundary conditions \eqref{boundary-Dir}.

\subsection{Dirichlet boundary conditions}\label{sec:Dir}
If $u$ is the solution of the initial boundary value problem \eqref{CahnHill+ine}-\eqref{initial}-\eqref{boundary-Dir}, then the function
\begin{equation}\label{u-Diri}
	\bar{u}(x,t):=\int_a^xu(y,t)\,dy-\frac1{b-a}\int_a^bdy\int_a^yu(\eta,t)\,d\eta
\end{equation}
satisfies the integrated version of \eqref{CahnHill+ine}
\begin{equation}\label{integrated}
	\tau \ub_{tt}+\bar{u}_t=-\e^2\ub_{xxxx}+W'(\bar{u}_x)_x,
\end{equation}
with Neumann boundary conditions
\begin{equation}\label{bound-ub}
	\ub_x(a,t)=\pm1, \quad \ub_x(b,t)=\pm1, \qquad \ub_{xxx}(a,t)=\ub_{xxx}(b,t)=0, \qquad \forall\,t\geq0,
\end{equation}
and initial data
\begin{equation*}
	\ub(x,0)=\ub_0(x), \quad \ub_t(x,0)=\ub_1(x), \qquad x\in[a,b].
\end{equation*}
Notice that for any function $u\in L^2(a,b)$, we can write $\ub=-y_x$, where $y$ is the solution to the elliptic problem
\begin{equation}\label{ell-dir}
	-y_{xx}=u \qquad \mbox{ in } \, (a,b), \qquad y(a)=y(b)=0.
\end{equation}
Then, the $L^2$--norm of $\ub$ can be seen as the $H^{-1}$--norm of $u$.

The first step to study the metastable dynamics is to show that $E_\e[u,\ub_t](t)$ is a non-increasing function of time $t$ 
along the solutions of \eqref{CahnHill+ine} with boundary conditions \eqref{boundary-Dir}.
\begin{lem}\label{lem:diss-Dir}
Let $(u,u_t)\in C([0,T],H^3(a,b)\times H^1(a,b))$ be solution of \eqref{CahnHill+ine}-\eqref{initial}-\eqref{boundary-Dir}.
If $E_\e$ is the functional defined in \eqref{energy}, then 
\begin{equation}\label{eq:energy-der}
	\frac{d}{dt}E_\e[u,\ub_t](t)=-\e^{-1}\|\ub_t(\cdot,t)\|^2_{{}_{L^2}},
\end{equation}
for any $t\in(0,T)$.
\end{lem}
\begin{proof}
Direct differentiation and integration by parts give
\begin{equation*}
	\frac{d}{dt}E_\e[u,\ub_t]=\int_a^b\left[\frac\tau\e\ub_t\ub_{tt}-\e u_{xx}u_t+\frac{W'(u)u_t}{\e}\right]dx+(\e u_xu_t)\bigg|_a^b.
\end{equation*}
The last term of the latter equality is zero because $u$ satisfies the Dirichlet boundary conditions \eqref{boundary-Dir}.
Using that $\ub_x=u$ and integrating by parts, we obtain
\begin{equation*}
	\e\frac{d}{dt}E_\e[u,\ub_t]=\int_a^b\left[\tau\ub_t\ub_{tt}+\e^2 \ub_{xxxx}\ub_t-W'(\ub_x)_x\ub_t\right]dx-(\ub_{xxx}\ub_t)\bigg|_a^b+(W'(\ub_x)\ub_t)\bigg|_a^b.
\end{equation*}
In this equality the last two terms are zero for the boundary conditions \eqref{bound-ub}.
Since $\ub$ satisfies the equation \eqref{integrated}, we end up with 
\begin{equation*}
	\e\frac{d}{dt}E_\e[u,\ub_t]=-\int_a^b(\ub_t)^2dx,
\end{equation*}
and the proof is complete.
\end{proof}
Therefore, from \eqref{eq:energy-der} it follows that the variation of the energy is related to the $L^2$--norm of $\ub_t$.
This is in contrast with the case of the hyperbolic Allen--Cahn equation (see \eqref{eq:energy-var-hyp}), 
where the variation is related to the $L^2$--norm of the velocity $u_t$.

Let us now define a $N$--transition layer structure initial datum for our problem. 
To this end, we fix a {\it piecewise constant function} as follows:
\begin{equation}\label{vstruct}
	\begin{aligned}
	v:[a,b]\rightarrow\{-1,+1\}\  \hbox{with $N$ jumps located at $a<h_1<h_2<\cdots<h_N<b$ and}\ r>0\\
	\hbox{such that}\ (h_i-r,h_i+r)\cap(h_j-r,h_j+r)=\emptyset \  \hbox{for}\ i\neq j\ \hbox{and}\   a\leq h_1-r,\ h_N+r\leq b.
	\end{aligned}
\end{equation}
Then, we assume that the initial data depend on $\e$, with
\begin{align}
	&\lim_{\varepsilon\rightarrow 0} \|u_0^\varepsilon-v\|_{{}_{L^1}}=0, \label{ass-u0}\\
&\label{energy-ini}
	E_\varepsilon[u_0^\varepsilon, \ub_1^\varepsilon]\leq Nc_0+\frac{1}{f(\e)},
\end{align}
for any $\varepsilon\ll 1$, where $f:(0,+\infty)\rightarrow(0,+\infty)$ and the positive constant $c_ 0$ is defined by \eqref{c_0}.
The condition \eqref{ass-u0} fixes the number of transitions and their relative positions in the limit $\varepsilon\to0$. 
The condition \eqref{energy-ini} requires that the energy at the time $t=0$ exceeds at most of $1/f(\e)$ the minimum possible to have these $N$ transitions.
In order to obtain slow evolution for the solutions, the function $1/f(\e)$ has to be small for small $\e$;
then, we assume that $f(\e)\to+\infty$ as $\e\to0^+$.

The next step is to prove a lower bound on the energy. 
In particular, since the variation of energy is related to the $L^2$--norm of $\ub_t$, we need to prove a lower bound on $P_\varepsilon[u]$  
under the assumption that $\ub$ is close to $\bar{v}$ in $L^1$.
\begin{prop}\label{prop:lower}
Assume that $W\in C^3(\R)$ satisfies \eqref{eq:ass-W} and define $\lambda:=\min\{W''(\pm1)\}$. 
Let $v$ as in \eqref{vstruct} and let $A\in(0,r\sqrt{2\lambda})$.
Then, there exist constants $C,\delta>0$ (depending only on $W,v$ and $A$) such that if $u\in H^1$ satisfies 
\begin{equation}\label{ass-ub}
	\|{\ub-\bar{v}}\|_{{}_{L^1}}\leq\delta,
\end{equation}
then
\begin{equation}\label{lower}
	P_\varepsilon[u]\geq Nc_0-C\exp(-A/\varepsilon).
\end{equation}
\end{prop}
\begin{proof}
Take $\hat r\in(0,r)$ and $\rho_1$ so small that $A\leq(r-\hat r)\sqrt{2\lambda-\nu\rho_1}$,
where $\nu:=\sup|W'''(x)|$ for $x\in[-1-\rho_1,1+\rho_1]$.
Then, choose $0<\rho_2 < \rho_1$  sufficiently small that
\begin{equation}\label{eq:forrho2}
\begin{aligned}
	\int_{1-\rho_1}^{1-\rho_2}\sqrt{2W(s)}\,ds&>\int_{1-\rho_2}^{1}\sqrt{2W(s)}\,ds,  \\
	\int_{-1+\rho_2}^{-1+\rho_1}\sqrt{2W(s)}\,ds&> \int_{-1}^{-1+\rho_2}\sqrt{2W(s)}\,ds.
	\end{aligned}
\end{equation}
Now, let us focus our attention on $h_i$, one of the transition points of $v$ and, to fix ideas, 
 let $v(h_i+r)= 1$ and $v(h_i-r)=-1$, the other case being analogous.
We claim that there exist $r_+$ and $r_-$ in $(0,\hat r)$ such that
\begin{equation}\label{2points}
	|u(h_i+r_+)-1|<\rho_2, \qquad \quad \mbox{ and } \qquad \quad |u(h_i-r_-)+1|<\rho_2.
\end{equation}
Indeed, we have
\begin{equation*}
	-\int_a^b (\ub-\bar{v})w'\,dx=\int_a^b(u-v)w\,dx,
\end{equation*}
for any test function $w\in C^1_c([a,b])$.
It follows that
\begin{equation*}
	\left|\int_a^b(u-v)w\,dx\right|\leq\|w'\|_{\infty}\int_a^b |\ub-\bar{v}|\,dx\leq\delta\|w'\|_{\infty},
\end{equation*} 
for any test function $w\in C^1_c([a,b])$.

Assume by contradiction that $|u-1|\geq\rho_2$ throughout $(h_i,h_i+\hat r)$. 
Being in particular $u - v$ continuous in $(h_i,h_i+\hat r)$, then  either $u-1\geq\rho_2>0$ or $u-1\leq-\rho_2<0$ in the whole interval under consideration. 
Therefore, choosing $w$ non negative with compact support contained in $(h_i,h_i+\hat r)$, we obtain
\begin{equation*}
	\rho_2\|w\|_{{}_{L^1}} \leq \int_{h_i}^{h_i+\hat r}|u-1|w\,dx=\left|\int_a^b(u-v)w\,dx\right|  \leq\delta\|w'\|_{\infty}.
\end{equation*}
This leads to a contradiction if $\delta$ is sufficiently small.
Similarly, one can prove the existence of $r_-\in(0,\hat r)$ such that $|u(h_i-r_-)+1|<\rho_2$.

Consider the interval $[h_i-r,h_i+r]$. 
Observe that from Young's inequality, it follows that for any $a\leq c<d\leq b$
\begin{equation}\label{eq:Young}
	\int_c^d\left[\frac{\e}{2}u_x^2+\frac{W(u)}{\e}\right]\,dx \geq \left|\int_{u(c)}^{u(d)}\sqrt{2W(s)}\,ds\right|.
\end{equation}
If $u(h_i+r_+)\geq1$ and $u(h_i-r_-)\leq-1$, then from \eqref{eq:Young}  we can conclude that
\begin{equation*}
	P_\e[u] \geq \int_{h_i-r_-}^{h_i+r_+}\left[\frac{\e}{2}u_x^2+\frac{W(u)}{\e}\right]\,dx \geq\int_{-1}^{1}\sqrt{2W(s)}\,ds=c_0
\end{equation*}
which leads to \eqref{lower} summing up all terms due to each transition point. 

On the other hand, notice that in general we have
\begin{align}
	\int_{h_i-r}^{h_i+r}\left[\frac{\e}{2}u_x^2+\frac{W(u)}{\e}\right]\,dx & 
	\geq \int_{h_i+r_+}^{h_i+r}\left[\frac{\e}{2}u_x^2+\frac{W(u)}{\e}\right]\,dx + \int_{h_i-r}^{h_i-r_-}\left[\frac{\e}{2}u_x^2+\frac{W(u)}{\e}\right]\,dx \notag \\
	& \quad +\int_{-1}^{1}\sqrt{2W(s)}\,ds-\int_{-1}^{u(h_i-r_-)}\sqrt{2W(s)}\,ds \notag \\
	& \quad-\int_{u(h_i+r_+)}^{1}\sqrt{2W(s)}\,ds. \label{eq:Pe}
\end{align}
To estimate the first term on the right hand side of \eqref{eq:Pe}, let us focus the attention on the interval $[h_i+r_+,h_i+r]$. 
Assume that $1-\rho_2<u(h_i+r_+)<1$ and consider the unique minimizer $z:[h_i+r_+,h_i+r]\rightarrow\R$ 
of the functional $$\int_{h_i+r_+}^{h_i+r}\left[\frac{\e}{2}u_x^2+\frac{W(u)}{\e}\right]\,dx$$ subject to the boundary condition 
$z(h_i+r_+)=u(h_i+r_+)$.
If the range of $z$ is not contained in the interval $(1-\rho_1,1+\rho_1)$, then from \eqref{eq:Young}, it follows that
\begin{equation}\label{E>fi}
	\int_{h_i+r_+}^{h_i+r}\left[\frac{\e}{2}z_x^2+\frac{W(z)}{\e}\right]\,dx>\int_{u(h_i+r_+)}^{1}\sqrt{2W(s)}\,ds.
\end{equation}
by the choice of $r_+$ and $\rho_2$. 
Suppose, on the other hand, that the range of $z$ is contained in the interval $(1-\rho_1,1+\rho_1)$. 
Then, the Euler--Lagrange equation for $z$ is
\begin{align*}
	z''(x)=\varepsilon^{-2}W'(z(x)), \quad \qquad x\in(h_i+r_+,h_i+r),\\
	z(h_i+r_+)=u(h_i+r_+), \quad \qquad z'(h_i+r)=0.
\end{align*}
Denoting by $\psi(x):=(z(x)-1)^2$, we have $\psi'=2(z-1)z'$ and 
\begin{equation*}
	\psi''(x)=2(z(x)-1)z''(x)+2z'(x)^2\geq\frac2{\varepsilon^2}(z(x)-1)W'(z(x)).
\end{equation*}
Since $|z(x)-1|\leq\rho_1$ for any $x\in[h_i+r_+,h_i+r]$, using Taylor's expansion 
\begin{equation*}
	W'(z(x))=W'(1)+W''(1)(z(x)-1)+R=W''(1)(z(x)-1)+R,
\end{equation*}
where $|R|\leq\nu|z-1|^2/2$, we obtain
\begin{align*}
	\psi''(x) & \geq \frac{2}{\varepsilon^2}W''(1)(z(x)-1)^2-\frac{\nu}{\varepsilon^2}|z(x)-1|^3\\
	& \geq \frac{2\lambda}{\varepsilon^2}(z(x)-1)^2-\frac{\nu\rho_1}{\varepsilon^2}(z(x)-1)^2\\
	& \geq \frac{\mu^2}{\varepsilon^2}\psi(x),
\end{align*}
where $\mu=A/(r-\hat r)$. 
Thus, $\psi$ satisfies
\begin{align*}
	\psi''(x)-\frac{\mu^2}{\varepsilon^2}\psi(x)\geq0, \quad \qquad x\in(h_i+r_+,h_i+r),\\
	\psi(h_i+r_+)=(u(h_i+r_+)-1)^2, \quad \qquad \psi'(h_i+r)=0.
\end{align*}
We compare $\psi$ with the solution $\hat \psi$ of
\begin{align*}
	\hat\psi''(x)-\frac{\mu^2}{\varepsilon^2}\hat\psi(x)=0, \quad \qquad x\in(h_i+r_+,h_i+r),\\
	\hat\psi(h_i+r_+)=(u(h_i+r_+)-1)^2, \quad \qquad \hat\psi'(h_i+r)=0,
\end{align*}
which can be explicitly calculated to be
\begin{equation*}
	\hat\psi(x)=\frac{(u(h_i+r_+)-1)^2}{\cosh\left[\frac\mu\varepsilon(r-r_+)\right]}\cosh\left[\frac\mu\varepsilon(x-(h_i+r))\right].
\end{equation*}
By the maximum principle, $\psi(x)\leq\hat\psi(x)$ so, in particular,
\begin{equation*}
	\psi(h_i+r)\leq\frac{(u(h_i+r_+)-1)^2}{\cosh\left[\frac\mu\varepsilon(r-r_+)\right]}\leq2\exp(-A/\varepsilon)(u(h_i+r_+)-1)^2.
\end{equation*}
Then, we have 
\begin{equation}\label{|z-v+|<exp}
	|z(h_i+r)-1|\leq\sqrt2\exp(-A/2\varepsilon)\rho_2.
\end{equation}
Now, by using Taylor's expansion for $W(s)$, we obtain
\begin{align*}
	W(s) & =W(1)+W'(1)(s-1) +\tfrac12W''(1)(s-1)^2+o(|s-1|^2)\\
	& \leq(s-1)^2\left(\frac{W''(1)}2+\frac{o(|s-1|^2)}{|s-1|^2}\right).
\end{align*}
Therefore, for $s$ sufficiently close to $1$ we have
\begin{equation}\label{W-quadratic}
	W(s)\leq\Lambda(s-1)^2.
\end{equation}
Using \eqref{|z-v+|<exp} and \eqref{W-quadratic}, we obtain
\begin{equation}\label{fi<exp}
	\left|\int_{z(h_i+r)}^{1}\sqrt{2W(s)}\,ds\right|\leq\sqrt{\Lambda/2}(z(h_i+r)-1)^2\leq\sqrt{2\Lambda}\,\rho_2^2\,\exp(-A/\varepsilon). 
\end{equation}
From \eqref{fi<exp} it follows that, for some constant $C>0$, 
\begin{align}
	\int_{h_i+r_+}^{h_i+r}\left[\frac{\e}{2}z_x^2+\frac{W(z)}{\e}\right]\,dx & \geq\left|\int_{z(h_i+r_+)}^{z(h_i+r)}\sqrt{2W(s)}\,ds\right| \nonumber\\
	& \geq \left|\int_{z(h_i+r_+)}^{1}\sqrt{2W(s)}\,ds-\int_{z(h_i+r)}^{1}\sqrt{2W(s)}\,ds\right| \nonumber\\
	& \geq \int_{u(h_i+r_+)}^{1}\sqrt{2W(s)}\,ds-\tfrac{C}{2N}\exp(-A/\varepsilon). \label{E>fi-exp}
\end{align}
Combining \eqref{E>fi} and \eqref{E>fi-exp}, we get that the constrained minimizer $z$ of the proposed variational problem satisfies
\begin{equation}\label{new1}	
	\int_{h_i+r_+}^{h_i+r}\left[\frac{\e}{2}z_x^2+\frac{W(z)}{\e}\right]\,dx\geq\int_{u(h_i+r_+)}^{1}\sqrt{2W(s)}\,ds-\tfrac{C}{2N}\exp(-A/\varepsilon).
\end{equation}
The restriction of $u$ to $[h_i+r_+,h_i+r]$ is an admissible function, so it must satisfy the same estimate
\begin{equation*}
	\begin{aligned}
	\int_{h_i+r_+}^{h_i+r}\left[\frac{\e}{2}u_x^2+\frac{W(u)}{\e}\right]\,dx&\geq \int_{h_i+r_+}^{h_i+r}\left[\frac{\e}{2}z_x^2+\frac{W(z)}{\e}\right]\,dx\\
	&\geq\int_{u(h_i+r_+)}^{1}\sqrt{2W(s)}\,ds-\tfrac{C}{2N}\exp(-A/\varepsilon).
	\end{aligned}
\end{equation*}

The second term on the right hand side of \eqref{eq:Pe} is estimated similarly by analyzing  the interval $[h_i-r,h_i-r_-]$ 
and using the second condition of \eqref{eq:forrho2} to obtain the corresponding inequality \eqref{E>fi}.
The obtained lower bound reads:
\begin{equation}\label{new2}	
	\int_{h_i-r}^{h_i-r_-}\left[\frac{\e}{2}z_x^2+\frac{W(z)}{\e}\right]\,dx\geq\int_{-1}^{u(h_i-r_-)}\sqrt{2W(s)}\,ds-\tfrac{C}{2N}\exp(-A/\varepsilon).
\end{equation}
Finally, by substituting \eqref{new1} and \eqref{new2} in \eqref{eq:Pe}, we deduce
\begin{equation*}
	\int_{h_i-r}^{h_i+r}\left[\frac{\e}{2}u_x^2+\frac{W(u)}{\e}\right]\,dx 
	 \geq c_0-\tfrac{C}N\exp(-A/\varepsilon).
\end{equation*}
Summing up all of these estimates for $i=1, \dots, N$, namely for all transition points, we end up with
\begin{equation*}
	P_\varepsilon[u]\geq\sum_{i=1}^N\int_{h_i-r}^{h_i+r}\left[\frac{\e}{2}u_x^2+\frac{W(u)}{\e}\right]\,dx\geq Nc_0-C\exp(-A/\varepsilon),
\end{equation*}
and the proof is complete.
\end{proof}

\begin{rem}\label{rem:ut}
As already underlined  in the previous proof, the assumption \eqref{ass-ub} of Proposition \ref{prop:lower} enables us to 
prove the existence of two points such that \eqref{2points} holds, and, once we have obtained this property,
the arguments used are the same of previous cases \cite{Grant,Folino2}.

It is worth to observe that for any $g\in L^p$
\begin{equation}\label{eq:gbarandg}
	\|\bar g\|_{{}_{L^p}}\leq2(b-a)\|g\|_{{}_{L^p}},
\end{equation}
and therefore \eqref{ass-ub} is indeed satisfied provided $u$ and $v$ are sufficiently close in $L^1$. 
Hence, from \eqref{eq:gbarandg} applied to $u_0^\varepsilon-v$ and from 
 \eqref{ass-u0}, we can apply Proposition \ref{prop:lower} and conclude \eqref{lower} for   the initial profile $u_0^\e$.
Finally, using also \eqref{energy-ini} we end up with
\begin{equation*}
	\tau\|\ub_1^\varepsilon\|_{{}_{L^2}}^2\leq C\e\left[\frac{1}{f(\e)}+\exp(-A/\varepsilon)\right].
\end{equation*}
\end{rem}

Proposition \ref{prop:lower} permits also to prove the following result.
\begin{prop}\label{prop:L2-norm}
Assume that $W\in C^3(\R)$ satisfies \eqref{eq:ass-W}.
Let $u^\varepsilon$ be the solution of \eqref{CahnHill+ine}-\eqref{initial}-\eqref{boundary-Dir} 
with initial data $u_0^{\varepsilon}$, $u_1^{\varepsilon}$ satisfying \eqref{ass-u0} and \eqref{energy-ini}.
Then, there exist positive constants $\varepsilon_0, C_1, C_2>0$ (independent on $\varepsilon$) such that
\begin{equation}\label{L2-norm}
	\int_0^{C_1\varepsilon^{-1}T_\e}\|\ub_t^\varepsilon\|^2_{{}_{L^2}}dt\leq C_2\varepsilon \left[\frac{1}{f(\e)}+\exp(-A/\varepsilon)\right],
\end{equation}
for all $\varepsilon\in(0,\varepsilon_0)$, where $T_\varepsilon:=\min\{f(\e),\exp(A/\varepsilon)\}$.
\end{prop}

\begin{proof}
Let $\varepsilon_0>0$ so small that for all $\varepsilon\in(0,\varepsilon_0)$, \eqref{energy-ini} holds and 
\begin{equation}\label{1/2delta}
	\|u_0^\varepsilon-v\|_{{}_{L^1}}\leq\frac14\delta(b-a)^{-1},
\end{equation}
where $\delta$ is the constant of Proposition \ref{prop:lower}. 
From \eqref{1/2delta} it follows that
\begin{equation}\label{1/2delta-int}
	\|\ub_0^\varepsilon-\bar{v}\|_{{}_{L^1}}\leq2(b-a)\|u_0^\varepsilon-v\|_{{}_{L^1}}\leq\frac12\delta.
\end{equation}
Let $\hat T>0$.
We claim that if
\begin{equation}\label{claim1}
	\int_0^{\hat T}\|\ub_t^\varepsilon\|_{{}_{L^1}}dt\leq\frac12\delta,
\end{equation}
then there exists $C>0$ such that
\begin{equation}\label{claim2}
	E_\varepsilon[u^\varepsilon, u_t^\varepsilon](\hat T)\geq Nc_0-C\exp(-A/\varepsilon).
\end{equation}
Indeed, $E_\varepsilon[u^\varepsilon, u_t^\varepsilon](\hat T)\geq P_\varepsilon[u^\varepsilon](\hat T)$ and 
inequality \eqref{claim2} follows from Proposition \ref{prop:lower} if $\|\ub^\varepsilon(\cdot,\hat T)-\bar{v}\|_{{}_{L^1}}\leq\delta$.
By using triangle inequality, \eqref{1/2delta-int} and \eqref{claim1}, we obtain
\begin{equation*}
	\|\ub^\varepsilon(\cdot,\hat T)-\bar{v}\|_{{}_{L^1}}\leq\|\ub^\varepsilon(\cdot,\hat T)-\ub_0^\varepsilon\|_{{}_{L^1}}+\|\ub_0^\varepsilon-\bar{v}\|_{{}_{L^1}}
	\leq\int_0^{\hat T}\|\ub_t^\varepsilon\|_{{}_{L^1}}+\frac12\delta\leq\delta.
\end{equation*}
By integrating the equality \eqref{eq:energy-der}, we deduce
\begin{equation}\label{dissipative}
	E_\e[u^\e_0,u^\e_1]-E_\e[u^\e,u_t^\e](\hat T)=\e^{-1}\int_0^{\hat T}\|\ub_t^\e\|^2_{{}_{L^2}}\,dt.
\end{equation}
Substituting  \eqref{energy-ini} and \eqref{claim2} in \eqref{dissipative}, one has 
\begin{equation}\label{L2-norm-Teps}
	\int_0^{\hat T}\|\ub_t^\varepsilon\|^2_{{}_{L^2}}dt\leq C_2\varepsilon\left[\frac{1}{f(\e)}+C\exp(-A/\varepsilon)\right].
\end{equation}
It remains to prove that inequality \eqref{claim1} holds for $\hat T\geq C_1\e^{-1}T_\varepsilon$.
If 
\begin{equation*}
	\int_0^{+\infty}\|\ub_t^\varepsilon\|_{{}_{L^1}}dt\leq\frac12\delta,
\end{equation*}
there is nothing to prove. 
Otherwise, choose $\hat T$ such that
\begin{equation*}
	\int_0^{\hat T}\|\ub_t^\varepsilon\|_{{}_{L^1}}dt=\frac12\delta.
\end{equation*}
Using H\"older's inequality and \eqref{L2-norm-Teps}, we infer
\begin{equation*}
	\frac12\delta\leq[\hat T(b-a)]^{1/2}\biggl(\int_0^{\hat T}\|\ub_t^\varepsilon\|^2_{{}_{L^2}}dt\biggr)^{1/2}\leq
	\left[\hat T(b-a)C_2\varepsilon\left(\frac{1}{f(\e)}+C\exp(-A/\varepsilon)\right)\right]^{1/2}.
\end{equation*}
It follows that there exists $C_1>0$ such that
\begin{equation*}
	\hat T\geq C_1\varepsilon^{-1}\min\{f(\e),\exp(A/\varepsilon)\},
\end{equation*}
and the proof is complete.
\end{proof}

Now, we can prove the main result concerning the slow evolution of the solutions in the case of Dirichlet boundary conditions.
\begin{thm}\label{main-scalar-Dir}
Assume that $W\in C^3(\R)$ satisfies \eqref{eq:ass-W} and define $\lambda:=\min\{W''(\pm1)\}$. 
Let $v$ be as in \eqref{vstruct} and let be a constant such that $A\in(0,r\sqrt{2\lambda})$. 
If $u^\varepsilon$ is the solution of \eqref{CahnHill+ine}-\eqref{initial}-\eqref{boundary-Dir} 
with initial data $u_0^{\varepsilon}$, $u_1^{\varepsilon}$ satisfying \eqref{ass-u0} and \eqref{energy-ini},
then, 
\begin{equation}\label{limit}
	\sup_{0\leq t\leq  T_\e}\|\ub^\varepsilon(\cdot,t)-\bar v\|_{{}_{L^1}}\xrightarrow[\varepsilon\rightarrow0]{}0,
\end{equation}
where $T_\varepsilon:=\min\{f(\e),\exp(A/\varepsilon)\}$.
Moreover, for any $0<\eta<1$,
\begin{equation}\label{limit-eta}
	\sup_{0\leq t\leq \e^\eta T_\e}\|u^\varepsilon(\cdot,t)-v\|_{{}_{L^1}}\xrightarrow[\varepsilon\rightarrow0]{}0.
\end{equation}
\end{thm}

\begin{proof}
We start by observing that triangle inequality gives
\begin{equation}\label{trianglebar}
	\|\ub^\varepsilon(\cdot,t)-\bar v\|_{{}_{L^1}}\leq\|\ub^\varepsilon(\cdot,t)-\ub_0^\varepsilon\|_{{}_{L^1}}+\|\ub_0^\varepsilon-\bar v\|_{{}_{L^1}},
\end{equation}
for all $t\in[0, T_\varepsilon]$. 
The last term of inequality \eqref{trianglebar} tends to $0$ by assumption \eqref{ass-u0} and \eqref{eq:gbarandg}; 
let us show that also the first one tends to 0 as $\e\to0$.
To this end, taking $\varepsilon$ so small that $C_1\varepsilon^{-1}\geq 1$, we can apply Proposition \ref{prop:L2-norm} and, by using \eqref{L2-norm}, we deduce that
\begin{align*}
	\|\ub^\e(\cdot,t)-\ub^\e_0\|^2_{{}_{L^1}} & \leq (b-a)\|\ub^\e(\cdot,t)-\ub^\e_0\|^2_{{}_{L^2}}=(b-a) \left\|\int_0^t\ub_s^\e(\cdot,s)\,ds\right\|^2_{L^2}\\
	& \leq (b-a) t\int_0^t\|\ub_t^\e\|^2_{{}_{L^2}}dt \leq (b-a) T_\varepsilon\int_0^{T_\varepsilon}\|\ub_t^\e\|^2_{{}_{L^2}}dt \leq 2C_2\e,
\end{align*}	
for all $t\in[0,  T_\varepsilon]$. Hence \eqref{limit} follows.

Moreover, fix $\eta\in(0,1)$, and, as before, 
\begin{equation}\label{triangle}
	\|u^\varepsilon(\cdot,t)- v\|_{{}_{L^1}}\leq\|u^\varepsilon(\cdot,t)-u_0^\varepsilon\|_{{}_{L^1}}+\|u_0^\varepsilon-v\|_{{}_{L^1}},
\end{equation}
for all $t\in[0,\e^\eta T_\varepsilon]$. As before,  from \eqref{ass-u0}, we need only to control the first term on the right hand side of  \eqref{triangle}.
Taking this time $\varepsilon$ so small that $C_1\varepsilon^{-1}\geq\e^\eta$, as before we obtain 
\begin{equation}\label{eq:bar}
	\|\ub^\e(\cdot,t)-\ub^\e_0\|^2_{{}_{L^2}}=\left\|\int_0^t\ub_s^\e(\cdot,s)\,ds\right\|^2_{L^2}\leq t\int_0^t\|\ub_t^\e\|^2_{{}_{L^2}}dt\leq 2C_2\e^{1+\eta},
\end{equation}	
for all $t\in[0,\e^\eta T_\varepsilon]$.

Denote by $w^\e(x,t):=u^\e(x,t)-u^\e_0(x)$ and $\bar{w}^\e(x,t):=\ub^\e(x,t)-\ub^\e_0(x)$.
Integrating by parts and using the boundary conditions \eqref{boundary-Dir}, we infer
\begin{equation}\label{ineq:w}
	\|w^\e(\cdot,t)\|^2_{{}_{L^2}}=\int_a^b\bar{w}_x^\e(x,t)w^\e(x,t)\,dx=-\int_a^b\bar{w}^\e(x,t)w^\e_x(x,t)\,dx\leq\|\bar{w}^\e(\cdot,t)\|_{{}_{L^2}}\|w^\e_x(\cdot,t)\|_{{}_{L^2}}.
\end{equation}
In order to estimate the last term of \eqref{ineq:w}, we use \eqref{eq:bar}, the assumption \eqref{energy-ini} and \eqref{dissipative}.
Indeed, since $\int_a^bu^\e_x(x,t)^2dx\leq C/\varepsilon$ for all $t\geq0$, we end up with
\begin{equation*}
	\|u^\e(\cdot,t)-u^\e_0\|^2_{{}_{L^2}}\leq C\e^{\eta/2},
\end{equation*}
for all $t\in[0,\e^\eta T_\varepsilon]$.
It follows that
\begin{equation*}
	\sup_{0\leq t\leq \e^\eta T_\varepsilon}\|u^\varepsilon(\cdot,t)-u_0^\varepsilon\|_{{}_{L^1}}\leq C\e^{\eta/4}.
\end{equation*}
Combining the latter estimate, \eqref{triangle} and by passing to the limit as $\varepsilon\to0$, we obtain \eqref{limit-eta}.
\end{proof}

\subsection{Neumann boundary conditions}\label{sec:Neu}
In the case of homogeneous Neumann boundary conditions \eqref{boundary-Neu}, we use a different primitive with respect to \eqref{u-Diri}.
If $u$ is the solution of the initial boundary value problem \eqref{CahnHill+ine}-\eqref{initial}-\eqref{boundary-Neu}, then the function
\begin{equation*}\label{u-Neu}
	\tilde{u}(x,t):=\int_a^xu(y,t)\,dy
\end{equation*}
satisfies the integrated version \eqref{integrated} with Dirichlet boundary conditions
\begin{equation}\label{bound-ut}
	\ut(a,t)=0, \quad \ut(b,t)=\int_a^bu(y,t)\,dy, \qquad \ut_{xx}(a,t)=\ut_{xx}(b,t)=0, \qquad \forall\,t\geq0,
\end{equation}
and initial data
\begin{equation*}
	\ut(x,0)=\int_a^xu_0(y)\,dy, \quad \ut_t(x,0)=\int_a^xu_1(y)\,dy, \qquad x\in[a,b].
\end{equation*}
Analogously to \eqref{ell-dir}, we observe that, if $u\in L^2(a,b)$ with $\displaystyle\int_a^b u\,dx=0$, then we can write $\ut=-y_x$, where $y_x$ satisfies 
\begin{equation*}
	-y_{xx}=u \qquad \mbox{ in } \, (a,b), \qquad y_x(a)=y_x(b)=0.
\end{equation*}

Let us consider the total mass $m(t):=\ut(b,t)=\displaystyle\int_a^b u(x,t)\,dx$.
A straightforward computation shows that the total mass $m$ satisfies the ordinary differential equation 
\begin{equation*}
	\tau m''(t)+m'(t)=0, \qquad m(0)=\int_a^b u_0(x)\,dx, \quad m'(0)=\int_a^b u_1(x)\,dx.
\end{equation*}
Thus, $m(t)=m(0)+\tau m'(0)(1-e^{-t/\tau})$ and the total mass $m$ is conserved if and only if 
\begin{equation}\label{ass-u1}
	\int_a^bu_1(y)\,dy=0.
\end{equation}
In other words, if $u_1$ satisfies \eqref{ass-u1}, then
\begin{equation}\label{conservation}
	m(t)\equiv M:=\int_a^bu_0(y)\,dy.
\end{equation}
This is a fundamental difference with respect to the classic Cahn--Hilliard equation, 
where the homogeneous Neumann boundary conditions implies \eqref{conservation}.
The conservation of the total mass \eqref{conservation} ensures the dissipative character of the equation \eqref{CahnHill+ine} 
with homogeneous boundary conditions \eqref{boundary-Neu}.
Indeed, if the initial velocity $u_1$ satisfies the assumption \eqref{ass-u1}, 
then the energy functional $E_\e[u,\ut]$ defined by \eqref{energy} decays along the orbits of \eqref{CahnHill+ine}-\eqref{boundary-Neu}.
\begin{lem}\label{lem:enNeu}
Let $(u,u_t)\in C([0,T],H^3(a,b)\times H^1(a,b))$ be the solution of \eqref{CahnHill+ine}-\eqref{initial}-\eqref{boundary-Neu}.
If $u_1$ satisfies \eqref{ass-u1}, then 
\begin{equation*}
	\frac{d}{dt}E_\e[u,\ut_t](t)=-\e^{-1}\|\ut_t(\cdot,t)\|^2_{{}_{L^2}},
\end{equation*}
for any $t\in(0,T)$.
\end{lem}
\begin{proof}
Let us proceed as in the proof of Lemma \eqref{lem:diss-Dir}:
\begin{equation*}
	\frac{d}{dt}E_\e[u,\ut_t]=\int_a^b\left[\frac\tau\e\ut_t\ut_{tt}-\e u_{xx}u_t+\frac{W'(u)u_t}{\e}\right]dx+(\e u_xu_t)\bigg|_a^b.
\end{equation*}
Here, the last term is zero because of the homogeneous Neumann boundary conditions \eqref{boundary-Neu}.
Using that $\ut_x=u$ and integrating by parts, we obtain
\begin{equation}\label{eq:forrem}
	\e\frac{d}{dt}E_\e[u,\ut_t]=\int_a^b\left[\tau\ut_t\ut_{tt}+\e^2 \ut_{xxxx}\ut_t-W'(\ut_x)_x\ut_t\right]dx-(\ut_{xxx}\ut_t)\bigg|_a^b+(W'(\ut_x)\ut_t)\bigg|_a^b.
\end{equation}
The boundary conditions \eqref{bound-ut} do not ensure that the last two terms vanish. 
However, if $u^\e_1$ satisfies the assumption \eqref{ass-u1} then $\ut_t(b,t)=0$ for any $t$. 
Therefore, we can conclude as in Lemma \ref{lem:diss-Dir} that
\begin{equation*}
	\e\frac{d}{dt}E_\e[u,\ut_t]=-\int_a^b(\ut_t)^2dx,
\end{equation*}
and the proof is complete.
\end{proof}

\begin{rem}\label{rem:nozeromean} 
The condition \eqref{ass-u1} is instrumental in the proof of the monotone decreasing nature of $E_\e[u,\ut_t](t)$ performed in the previous lemma. 
However, for our final aim of showing the slow evolution of solutions to \eqref{CahnHill+ine}-\eqref{boundary-Neu}, 
a sufficiently small increase of the energy is allowed. 
This property can be proved even if \eqref{ass-u1} is violated, provided the initial velocity $u_1$ is sufficiently small, 
and one knows a control of the boundary values of the solution. 
Indeed, from \eqref{eq:forrem} it follows
\begin{equation*}
	\e\frac{d}{dt}E_\e[u,\ut_t]=-\int_a^b(\ut_t)^2dx-\left(u_{xx}(b,t)+W'(u(b,t))\right)\ut_t(b,t).
\end{equation*}
Using that $\ut_t(b,t)=m'(t)=m'(0)e^{-t/\tau}$ and integrating we obtain
\begin{align*}
	E_\e[u,\ut_t](T)-E_\e[u,\ut_t](0)= & -\e^{-1}\int_0^T\|\ut_t(\cdot,t)\|^2_{{}_{L^2}}\,dt\\
	&+\e^{-1}m'(0)\int_0^T\left(-u_{xx}(b,t)+W'(u(b,t))\right)e^{-t/\tau}\,dt.
\end{align*}
Assuming that the last integral is uniformly bounded in time, we deduce that
\begin{equation*}
	E_\e[u,\ut_t](T)+\e^{-1}\int_0^T\|\ut_t(\cdot,t)\|^2_{{}_{L^2}}\,dt\leq E_\e[u,\ut_t](0)+C\e^{-1}|m'(0)|,
\end{equation*}
for any $T>0$.
Therefore, in general the energy is not decreasing, but if the term $\e^{-1}|m'(0)|$ is very small 
we can also obtain slow evolution for the solutions applying the energy approach. 
\end{rem}

We now proceed in the proof of the slow motion of the solutions in the case of homogeneous Neumann boundary conditions under the extra zero--mean assumption \eqref{ass-u1}.
Indeed, this assumption enables us to use the same strategy of the case of Dirichlet boundary conditions, in particular because all boundary terms vanish.

The lower bound \eqref{lower} holds true also if a function $\ut$ is close enough to $\tilde{v}$ in $L^1$. 
Indeed, by reasoning as in the proof of Proposition \ref{prop:lower}, we can prove the existence of two points such that \eqref{2points} holds. 
Then, in the same way of Proposition \ref{prop:lower} we obtain that there exists $\delta>0$ such that
\begin{equation}\label{lower-Neu}
	\|\ut-\tilde{v}\|_{{}_{L^1}}\leq\delta \qquad \quad \Longrightarrow \qquad \quad P_\varepsilon[u]\geq Nc_0-C\exp(-A/\varepsilon).
\end{equation}
Noting  that an analogous of \eqref{eq:gbarandg} also holds true for $\tilde g$, using \eqref{lower-Neu} and reasoning as the case of Dirichlet boundary conditions,
 we can prove the slow evolution also in the case of homogeneous Neumann boundary conditions. 
\begin{thm}\label{main-scalar-Neu}
Assume that $W\in C^3(\R)$ satisfies \eqref{eq:ass-W} and define $\lambda:=\min\{W''(\pm1)\}$. 
Let $v$ be as in \eqref{vstruct} and let be a constant such that $A\in(0,r\sqrt{2\lambda})$.
In addition, let $u^\varepsilon$ be the solution of \eqref{CahnHill+ine}-\eqref{initial}-\eqref{boundary-Neu} 
with initial data $u_0^{\varepsilon}$, $u_1^{\varepsilon}$ satisfying \eqref{ass-u0} and 
\begin{equation}\label{energy-ini-Neu}
	E_\e[u_0^\e,\ut_1^\e]\leq Nc_0+\frac{1}{f(\e)},
\end{equation}
for any $\e\ll 1$, and $c_0$ and $f$ as in \eqref{energy-ini}.
Moreover, assume that $u^\e_1$ satisfies \eqref{ass-u1}.
Then, 
\begin{equation*}\label{limitNeu}
	\sup_{0\leq t\leq  T_\e}\|\ut^\varepsilon(\cdot,t)-\tilde v\|_{{}_{L^1}}\xrightarrow[\varepsilon\rightarrow0]{}0,
\end{equation*}
where $T_\varepsilon:=\min\{f(\e),\exp(A/\varepsilon)\}$.
Moreover, for any $0<\eta<1$,
\begin{equation*}\label{limit-etaNeu}
	\sup_{0\leq t\leq \e^\eta T_\e}\|u^\varepsilon(\cdot,t)-v\|_{{}_{L^1}}\xrightarrow[\varepsilon\rightarrow0]{}0.
\end{equation*}
\end{thm}

\subsection{Layer dynamics}\label{sec:layer}
Theorems \ref{main-scalar-Dir} and \ref{main-scalar-Neu} state that the solutions of \eqref{CahnHill+ine} with either Dirichlet or Neumann boundary conditions 
and appropriate initial data evolve very slow in time and the profile $u^\e$ maintains the same $N$--transition layer structure of the initial profile $u_0^\e$. 
Using a standard procedure (see \cite{Bron-Kohn,Grant,Folino,Folino2}), it is possible to show that Theorem \ref{main-scalar-Dir} 
(Theorem \ref{main-scalar-Neu} in the case of homogeneous Neumann boundary conditions) implies that the transition points $h_1,\dots,h_N$ move with a very small velocity. 
This result is obtained in the same way for either Dirichlet or Neumann boundary conditions. 
For definiteness, we consider the case of Dirichlet boundary conditions \eqref{boundary-Dir}.

Fix a piecewise constant function $v$ as in \eqref{vstruct}; its {\it interface} $I[v]$ is defined by 
\begin{equation*}
	I[v]:=\{h_1,h_2,\ldots,h_N\}.
\end{equation*}
For an arbitrary function $u:[a,b]\rightarrow\mathbb{R}$ and an arbitrary closed subset $K\subset\R\backslash\{\pm1\}$,
the {\it interface} $I_K[u]$ is defined by
\begin{equation*}
	I_K[u]:=u^{-1}(K).
\end{equation*}
Finally, for any $A,B\subset\mathbb{R}$ the {\it Hausdorff distance} $d(A,B)$ between $A$ and $B$ is defined by 
\begin{equation*}
	d(A,B):=\max\biggl\{\sup_{\alpha\in A}d(\alpha,B),\,\sup_{\beta\in B}d(\beta,A)\biggr\},
\end{equation*}
where $d(\beta,A):=\inf\{|\beta-\alpha|: \alpha\in A\}$. 

The following result is purely variational in character and states that, if a function $u\in H^1([a,b])$ is such that $\ub$ is close to $\bar{v}$ in $L^1$ and 
$P_\varepsilon[u]$ exceeds of a small quantity the minimum energy to have $N$ transitions, then 
the distance between the interfaces $I_K[u]$ and $I_K[v]$ is small.  
\begin{lem}\label{lem:interface}
Assume that $W\in C^3(\R)$ satisfies \eqref{eq:ass-W} and let $v$ be as in \eqref{vstruct}.
Given $\delta_1\in(0,r)$ and a closed subset $K\subset\R\backslash\{\pm1\}$, 
there exist constants $\hat\delta,\varepsilon_0,\Gamma>0$ (independent on $\e$) such that for any $u\in H^1([a,b])$ satisfying
\begin{equation}\label{eq:lem-interf}
	\|\ub-\bar v\|_{{}_{L^1}}<\hat\delta \qquad \quad \mbox{ and } \qquad \quad P_\varepsilon[u]\leq Nc_0+\Gamma,
\end{equation}
for all $\varepsilon\in(0,\varepsilon_0)$, we have
\begin{equation}\label{lem:d-interfaces}
	d(I_K[u], I[v])<\tfrac12\delta_1.
\end{equation}
\end{lem}
\begin{proof}
Fix $\delta_1\in(0,r)$ and choose $\rho>0$ small enough that 
\begin{equation*}
	I_\rho:=(-1-\rho,-1+\rho)\cup(1-\rho,1+\rho)\subset\R\backslash K, 
\end{equation*}
and 
\begin{equation*}
	\inf\left\{\left|\int_{\xi_1}^{\xi_2}\sqrt{2W(s)}\,ds\right| : \xi_1\in K, \xi_2\in I_\rho\right\}>2\Gamma,
\end{equation*}
where
\begin{equation*}
	\Gamma:=2N\max\left\{\int_{1-\rho}^{1}\sqrt{2W(s)}\,ds, \, \int_{-1}^{-1+\rho}\sqrt{2W(s)}\,ds \right\}.
\end{equation*}
By reasoning as in the proof of \eqref{2points} in Proposition \ref{prop:lower}, we can prove that for each $i$ there exist
\begin{equation*}
	x^-_{i}\in(h_i-\delta_1/2,h_i) \qquad \textrm{and} \qquad x^+_{i}\in(h_i,h_i+\delta_1/2),
\end{equation*}
such that
\begin{equation*}
	|u(x^-_{i})-v(x^-_{i})|<\rho \qquad \textrm{and} \qquad |u(x^+_{i})-v(x^+_{i})|<\rho.
\end{equation*}
Suppose that \eqref{lem:d-interfaces} is violated. 
Using \eqref{eq:Young}, we deduce
\begin{equation}\label{diseq:E1}
	P_\varepsilon[u]\geq\sum_{i=1}^N\left|\int_{u(x^-_{i})}^{u(x^+_{i})}\sqrt{2W(s)}\,ds\right|+\inf\left\{\left|\int_{\xi_1}^{\xi_2}\sqrt{2W(s)}\,ds\right| : \xi_1\in K, \xi_2\in I_\rho\right\}. 
\end{equation}
On the other hand, we have
\begin{align*}
	\left|\int_{u(x^-_{i})}^{u(x^+_{i})}\sqrt{2W(s)}\,ds\right|&\geq\int_{-1}^{1}\sqrt{2W(s)}\,ds-\int_{-1}^{-1+\rho}\sqrt{2W(s)}\,ds-\int_{1-\rho}^{1}\sqrt{2W(s)}\,ds\\
	&\geq\int_{-1}^{1}\sqrt{2W(s)}\,ds-\frac{\Gamma}{N}. 
\end{align*}
Substituting the latter bound in \eqref{diseq:E1} and recalling the definition of $c_0$ \eqref{c_0}, we infer
\begin{equation*}
	P_\varepsilon[u]\geq Nc_0-\Gamma+\inf\left\{\left|\int_{\xi_1}^{\xi_2}\sqrt{2W(s)}\,ds\right| : \xi_1\in K, \xi_2\in I_\rho\right\}.
\end{equation*}
For the choice of $\rho$ and assumption \eqref{eq:lem-interf}, we obtain
\begin{align*}
	P_\varepsilon[u]>Nc_0+\Gamma\geq P_\varepsilon[u],
\end{align*}
which is a contradiction. Hence, the bound \eqref{lem:d-interfaces} is true.
\end{proof}

Theorem \ref{main-scalar-Dir} and Lemma \ref{lem:interface} permit to prove the following result on the motion of the transition points $h_1,\ldots,h_N$.
\begin{thm}\label{thm:interface}
Assume that $W\in C^3(\R)$ satisfies \eqref{eq:ass-W}.
Let $u^\varepsilon$ be the solution of \eqref{CahnHill+ine}-\eqref{initial}-\eqref{boundary-Dir} 
with initial data $u_0^{\varepsilon}$, $u_1^{\varepsilon}$ satisfying \eqref{ass-u0} and \eqref{energy-ini}. 
Given $\delta_1\in(0,r)$ and a closed subset $K\subset\R\backslash\{\pm1\}$, set
\begin{equation*}
	t_\varepsilon(\delta_1)=\inf\{t:\; d(I_K[u^\varepsilon(\cdot,t)],I_K[u_0^\varepsilon])>\delta_1\}.
\end{equation*}
There exists $\varepsilon_0>0$ such that if $\varepsilon\in(0,\varepsilon_0)$ then
\begin{equation*}
	t_\varepsilon(\delta_1)>\min\{f(\e),\exp(A/\varepsilon)\}.
\end{equation*}
\end{thm}
\begin{proof}
Let $\varepsilon_0>0$ so small that the assumptions on the initial data \eqref{ass-u0}, \eqref{energy-ini} 
imply that $u_0^\varepsilon$ satisfy \eqref{eq:lem-interf} for all $\varepsilon\in(0,\varepsilon_0)$.
From Lemma \ref{lem:interface} it follows that
\begin{equation}\label{interfaces-u0}
	d(I_K[u_0^\varepsilon], I[v])<\tfrac12\delta_1.
\end{equation}
Now, we apply the same reasoning to $u^\varepsilon(\cdot,t)$ for all $t\leq\min\{f(\e),\exp(A/\varepsilon)\}$.
Assumption \eqref{eq:lem-interf} is satisfied thanks to \eqref{limit} in Theorem \ref{main-scalar-Dir},
and because $E_\varepsilon[u^\varepsilon,\ub^\varepsilon_t](t)$ is a non-increasing function of $t$.
Then,
\begin{equation}\label{interfaces-u}
	d(I_K[u^\varepsilon(t)], I[v])<\tfrac12\delta_1
\end{equation}
for all $t\in(0,\min\{f(\e),\exp(A/\varepsilon)\})$. 
Combining \eqref{interfaces-u0} and \eqref{interfaces-u}, we obtain
\begin{equation*}
	d(I_K[u^\varepsilon(t)],I_K[u_0^\varepsilon])<\delta_1
\end{equation*}
for all $t\in(0,\min\{f(\e),\exp(A/\varepsilon)\})$ and the proof is complete.
\end{proof}
In particular, let us assume that the function $f(\e)\geq\exp(A/\e)$ and so, 
that the energy at $t=0$ is exponentially close to $Nc_0$, the minimum energy to have $N$ transitions between $-1$ and $+1$.
Thanks to Theorems \ref{main-scalar-Dir} and \ref{thm:interface}, we obtain exponentially slow motion for solutions of \eqref{CahnHill+ine} 
with boundary conditions \eqref{boundary-Dir} and appropriate initial data: 
the solutions maintain the same structure of the initial profile at least for an exponentially long time
and the transition points move with exponentially small velocity.

\subsection{Construction of $N$--transition layer structure functions}\label{sec:N-translayer}
Fix a piecewise constant function $v$ as in \eqref{vstruct} and assume without loss of generality $v=-1$ in $(a,h_1)$.
As we have already mentioned in Remark \ref{rem:ut}, the assumptions on the initial data \eqref{ass-u0} and \eqref{energy-ini} 
(or \eqref{energy-ini-Neu} in the case of Neumann boundary conditions) are equivalent to
\begin{equation}\label{eq:trans}
	\lim_{\varepsilon\rightarrow 0} \|u_0^\varepsilon-v\|_{{}_{L^1}}=0, \qquad 
	P_\varepsilon[u_0^\varepsilon]\leq Nc_0+\frac{1}{f(\e)},
\end{equation}
for the initial profile $u_0^\e$ and 
\begin{equation}\label{eq:ut-ub}
	\tau\|\ub_1^\varepsilon\|_{{}_{L^2}}^2\leq C\e\left[\frac{1}{f(\e)}+\exp(-A/\varepsilon)\right], \qquad 
	\left(\mbox{or } \,\tau\|\ut_1^\varepsilon\|_{{}_{L^2}}^2\leq C\e\left[\frac{1}{f(\e)}+\exp(-A/\varepsilon)\right]\right).
\end{equation}
The latter condition is automatically satisfied if 
\begin{equation*}
	\tau\|u_1^\varepsilon\|_{{}_{L^2}}^2\leq C\e\left[\frac{1}{f(\e)}+\exp(-A/\varepsilon)\right].
\end{equation*}
Let us construct an example of family of functions $u_0^\e$ satisfying \eqref{eq:trans}.
Define $u_0^\varepsilon:=v$ outside of $\bigcup_{i=1}^N (h_i-r,h_i+r)$ and in each interval $(h_i-r,h_i+r)$ 
we use particular stationary solutions of \eqref{CahnHill+ine}:
consider the unique solution to the boundary value problem 
\begin{equation*}
	-\varepsilon^2\Phi_\e''+W'(\Phi_\e)=0, \qquad
	\Phi_\e(0)=0, \qquad
	\Phi_\e(x)\rightarrow\pm1 \quad \mbox{ as } \quad x\rightarrow\pm\infty.
\end{equation*}
The existence of a unique solution is guaranteed by the assumption \eqref{eq:ass-W} on $W$.
We can rewrite the solution as $\Phi_\e(x)=\omega(x/\varepsilon)$, where $\omega$ satisfies
\begin{equation*}
	-\omega''+W'(\omega)=0, \qquad
	\omega(0)=0, \qquad
	\omega(x)\rightarrow\pm1 \quad \mbox{ as } \quad x\rightarrow\pm\infty,
\end{equation*}
or, equivalently
\begin{equation*}
	\omega'=\sqrt{2W(\omega)}, \qquad \quad \omega(0)=0. 
\end{equation*}
In the simplest example $W(u)=\frac14(u^2-1)^2$, we can compute explicitly $\omega$ and we have $\omega(x)=\tanh(x/\sqrt2)$.
For $i=1,\cdots, N$, define $u^\e_0$ in $(h_i-r,h_i+r)$ as 
\begin{equation*}
	u_0^\varepsilon(x):=\begin{cases}
	\vspace{.2cm}
	(-1)^i+\frac{x-h_i+r}{\varepsilon}\left\{\Phi_\e\bigl((-r+\e)(-1)^{i+1}\bigr)-(-1)^i\right\},&  x\in(h_i-r,h_i-r+\varepsilon),\\
	\vspace{.2cm}
	\Phi_\e\bigl((x-h_i)(-1)^{i+1}\bigr) & x\in[h_i-r+\varepsilon,h_i+r-\varepsilon]\\
	(-1)^{i+1}+\frac{h_i+r-x}{\varepsilon}\left\{\Phi_\e\bigl((r-\e)(-1)^{i+1}\bigr)-(-1)^{i+1}\right\}, &  x\in(h_i+r-\varepsilon,h_i+r). 
	\end{cases}
\end{equation*}
The function $u_0^\e\in H^1([a,b])$ satisfies both Neumann \eqref{boundary-Neu} and Dirichlet \eqref{boundary-Dir} boundary conditions.
Let us verify that $u_0^\e$ satisfies \eqref{eq:trans}.
It is easy to check the $L^1$ requirement;
therefore, we focus the attention on the energy requirement.
Consider an interval $[h_i-r,h_i+r]$ (for example $i$ odd); we have
\begin{equation*}
	\int_{h_i-r}^{h_i+r}\left[\frac\varepsilon2{(u_0^\e)}_x^2+\frac{W(u^\e_0)}\varepsilon\right]dx=I_1+I_2+I_3.
\end{equation*}
where 
\begin{align*}
	I_1&:=\int_{h_i-r}^{h_i-r+\varepsilon}\left[\frac\varepsilon2 {(u_0^\varepsilon)}_x^2+\frac{W(u_0^\varepsilon)}\varepsilon\right]dx,\\
	I_2&:=\int_{h_i-r+\varepsilon}^{h_i+r-\varepsilon}\left[\frac\varepsilon2 {(u_0^\varepsilon)}_x^2+\frac{W(u_0^\varepsilon)}\varepsilon\right]dx,\\
	I_3&:=\int_{h_i+r-\varepsilon}^{h_i+r}\left[\frac\varepsilon2 {(u_0^\varepsilon)}_x^2+\frac{W(u_0^\varepsilon)}\varepsilon\right]dx.		
\end{align*}
Let us estimate these three term. 
Concerning $I_2$, since $\frac{\varepsilon^2}2\left(\Phi_\e'\right)^2=W(\Phi_\e)$, we get
\begin{equation*}
	I_2=\int_{-r+\e}^{r-\e}\Phi_\e'(x)\sqrt{2W(\Phi_\e(x))}\,dx=\int_{\Phi_\e(-r+\varepsilon)}^{\Phi_\e(r-\varepsilon)}\sqrt{2W(s)}\,ds\leq c_0.  
\end{equation*}
For $I_3$, we have
\begin{equation*}
	I_3:=\int_{r-\e}^{r}\left[\frac1{2\varepsilon}\left(\Phi_\e(r-\e)-1\right)^2+\frac{1}{\varepsilon}W\Bigl(1+\frac{r-x}{\varepsilon}\left(\Phi_\e(r-\e)-1\right)\Bigr)\right]dx.
\end{equation*}
To estimate the latter term, for $\varepsilon$ sufficiently small, we use \eqref{W-quadratic} to obtain
\begin{equation*}
	W\Bigl(1+\frac{r-x}{\varepsilon}\left(\Phi_\e(r-\e)-1\right)\Bigr)\leq\Lambda(\Phi_\e(r-\e)-1)^2.
\end{equation*}
Hence, we infer
\begin{equation}\label{diseq:I1}
	I_3\leq C(\omega(r/\varepsilon-1)-1)^2.
\end{equation}
Here and in what follows, $C$ is a positive constant (independent on $\varepsilon$) whose value may change from line to line.
Since $\omega(x)\to 1$ as $x\to+\infty$, there exists $x_1>0$ sufficiently large so that
\begin{equation*}
	\omega'(x)\geq\sqrt{\frac\lambda2}(1-\omega(x)),
\end{equation*}
for all $x\geq x_1$, where $\lambda:=\min\{W''(\pm1)\}$.
Therefore,
\begin{equation}\label{w-a}
	1-\omega(x)\leq C\exp\left(-\sqrt{\lambda/2}\,x\right),
\end{equation}
for all $x\geq x_1$. 
If $\varepsilon$ is so small that $r/\varepsilon-1\geq x_1$, by substituting \eqref{w-a} into \eqref{diseq:I1}, we obtain
\begin{equation*}
	I_3\leq C\exp\left(-\sqrt{\lambda/2}(r/\varepsilon-1)\right)\leq C\exp(-A/\varepsilon),
\end{equation*}
for all positive constant $A\leq r\sqrt{2\lambda}$.
In a similar way, we can obtain the estimate for $I_1$.
Combining the estimates for $I_1$, $I_2$ and $I_3$ and summing up for $i=1,\dots, N$, we get 
$$P_\varepsilon[u_0^\varepsilon]\leq Nc_0+C\exp(-A/\varepsilon).$$
Observe that $u_0^\varepsilon\in W^{1,\infty}(a,b)$ and, since $P_\e$ is continuous on this space 
and $W^{1,\infty}(a,b)$ functions can be approximated arbitrarily closely by $C^k$ functions (for arbitrarily large $k$), 
we can construct an arbitrarily smooth function, equal to $v$ outside of $\bigcup_{i=1}^N (h_i-r,h_i+r)$,
which satisfies the condition \eqref{eq:trans}.

\section{The case of systems}\label{sec:sys}
The results of the previous section carry over to systems of \emph{hyperbolic Cahn--Hilliard equations}:
\begin{equation}\label{system-CH}
	\tau \bm{u}_{tt}+\bm{u}_t=(-\varepsilon^2\bm{u}_{xx}+\nabla W(\bm{u}))_{xx}, \qquad \quad x\in [a,b], \; t>0,
\end{equation}
where $\bm{u}(x,t)\in\R^m$ is a vector-valued function, $W:\R^m\rightarrow\R$ and $\varepsilon,\tau$ are positive parameters.
As in the scalar case, system \eqref{system-CH} is a hyperbolic variation of the system of Cahn--Hilliard equations 
\begin{equation}\label{Cahn-Morral}
	\bm{u}_t=(-\varepsilon^2\bm{u}_{xx}+\nabla W(\bm{u}))_{xx}. 
\end{equation}
As previously mentioned in the scalar case \eqref{CahnHill}, the Cahn--Hilliard equation was proposed to model phase separation in binary alloy;
if we consider $m+1$ components we obtain system \eqref{Cahn-Morral}, which describes the phase separation of a mixture with $m+1$ components.
The study of system \eqref{Cahn-Morral} was initiated in \cite{MorralCahn} and \cite{Defontaine1,Defontaine2}.
We also refer to \cite{Eyre} and references therein, where in particular  the differences between multicomponent and binary alloys are addressed.
Finally, we recall that the metastable properties of the solutions have been studied in \cite{Grant}.

Here, using the same approach of the previous section, we study the metastable properties of the solutions to \eqref{system-CH}.
For definiteness, we shall only refer to the case when \eqref{system-CH} is complemented with homogeneous Neumann boundary conditions
\begin{equation}\label{Neumann-sys}
	\bm{u}_x(a,t)=\bm{u}_x(b,t)=0, \qquad\quad \forall \,t>0.
\end{equation}
For the case of Dirichlet boundary conditions, which fixes two different zeros of $W$, the proofs can be adapted as for the scalar case discussed above.
Finally, we consider system \eqref{system-CH} subject to the initial data 
\begin{equation}\label{cond-iniz-sys}
	\bm{u}(x,0)=\bm{u}_0(x), \qquad \bm{u}_t(x,0)=\bm{u}_1(x), \qquad \qquad x\in[a,b].
\end{equation} 
We assume that $W\in C^3(\R^m,\R)$ is a positive function with global minimum equal to $0$ 
reached only at a finite number of points, namely there exist $K>1$ and $\bm{z}_1,\dots,\bm{z}_K\in\R^m$ such that
\begin{equation}\label{W1}
	W(\bm{u})\geq0 \quad \forall \, \bm{u}\in\R^m, \quad  \textrm{ and } \quad W(\bm{u})=0 \Longleftrightarrow \bm{u}\in\{\bm{z}_1,\dots,\bm{z}_K\}.
\end{equation}
Moreover, we assume that the Hessian $\nabla^2W$ is positive definite at each zero of $W$:
\begin{equation}\label{W2}
	\nabla^2W(\bm{z}_j)\bm{v}\cdot\bm{v}>0 \quad \textrm{ for all } \, j=1,\dots,K \, \textrm{ and }  \, \bm{v}\in\R^m\backslash\{\bm 0\}. 
\end{equation}
For $j=1,\dots,K$, denote by $\lambda_j$ the minimum of the eigenvalues of $\nabla^2 W(\bm{z}_j)$ and by $\lambda_0=\displaystyle\min_j\lambda_j$.

Similar to the scalar case, we define the energy functional
\begin{equation}\label{energy-sys}
	E_\varepsilon[\bm{u},\bm{\ut}_t](t):=\frac{\tau}{2\varepsilon}\|\bm{\ut}_t(\cdot,t)\|^2_{{}_{L^2}}+P_\varepsilon[\bm{u}](t),
\end{equation}
where 
\begin{equation*}
	\bm{\ut}(x,t):=\int_a^x \bm{u}(y,t)\,dy, \mbox{ and } \quad 
	P_\varepsilon[\bm{u}](t):=\int_a^b\left[\frac\varepsilon2 |\bm{u}_x(x,t)|^2+\frac{W(\bm{u}(x,t))}\varepsilon\right]dx.
\end{equation*}
If the initial velocity satisfies
\begin{equation}\label{ass-u1-sys}
	\int_a^b\bm{u}_1(x)\,dx=0,
\end{equation}
then the total mass is conserved:
\begin{equation*}
	\bm{M}(t):=\int_a^b \bm{u}(x,t)\,dx\equiv\int_a^b \bm{u}_0(x)\,dx,
\end{equation*}
and the energy functional is a non-increasing function of time $t$ along the solutions to \eqref{system-CH} with boundary conditions \eqref{Neumann-sys}:
\begin{equation}\label{eq:energy-der-sys}
	E_\e[\bm{u},\bm{\ut}_t](0)-E_\e[\bm{u},\bm{\ut}_t](T)=\e^{-1}\int_0^T\|\bm{\ut}_t(\cdot,t)\|^2_{{}_{L^2}}dt,
\end{equation}
for any $T>0$.

Fix $\bm{v}:[a,b]\rightarrow\{\bm{z}_1,\dots,\bm{z}_K\}$ having exactly $N$ jumps located at $a<h_1<h_2<\cdots<h_N<b$, and $r>0$ as in 
\eqref{vstruct}.
In the scalar case the minimum energy to have a transition between the two equilibrium points $-1$ and $+1$ is the constant $c_0$ introduced in \eqref{c_0}.
In the case of systems, from Young's inequality and the positivity of the term $\frac{\tau}{2\varepsilon}\|\bm{\ut}_t\|^2_{{}_{L^2}}$, it follows that
\begin{equation}\label{E-Young}
	E_\varepsilon[\bm{u},\bm{\ut}_t](t)\geq P_\varepsilon[\bm{u}](t)\geq\sqrt2\int_a^b\sqrt{W(\bm{u}(x,t))}|\bm{u}_x(x,t)|dx.
\end{equation}
This justifies the use of the modified energy \eqref{energy-sys}; 
indeed, the right hand side of inequality \eqref{E-Young} is strictly positive and does not depend on $\varepsilon$. 
For \eqref{E-Young}, we assign to the discontinuous function $\bm{v}$ the asymptotic energy
\begin{equation*}
	P_0[\bm{v}]:=\sum_{i=1}^N\phi(\bm{v}(h_i-r),\bm{v}(h_i+r)),
\end{equation*}
where
\begin{equation*}
\phi(\bm{\xi}_1,\bm{\xi}_2):=\inf\left\{J[\bm{z}]: \bm{z}\in \emph{AC}([a,b],\R^m), \bm{z}(a)=\bm{\xi}_1, \bm{z}(b)=\bm{\xi}_2\right\}
\end{equation*}
and
\begin{equation*}
	J[\bm{z}]:=\sqrt2\int_a^b\sqrt{W(\bm{z}(s))}|\bm{z}'(s)|ds.
\end{equation*}	
In other words, $\phi(\bm{\xi}_1,\bm{\xi}_2)$ represents the infimum of the line integral $\displaystyle{ \int_\gamma  \sqrt{2W} d\sigma}$ among all absolutely continuous curves connecting the two points $\bm{\xi}_1$ and $\bm{\xi}_2$  in $\R^m$.

The lower bound for the energy of  Proposition \ref{prop:lower} is extended to the present case just using the needed vector notation and definitions; 
for details see \cite{Grant,Folino2}.
\begin{prop}\label{prop:lower-sys}
Assume that $W:\R^m\rightarrow\R$ satisfies \eqref{W1}-\eqref{W2}.
Let $\bm{v}:[a,b]\rightarrow\{\bm{z}_1,\dots,\bm{z}_K\}$ as defined before and let $A$ be a positive constant less than $r\sqrt{2\lambda_0}$. 
Then, there exist constants $C,\delta>0$ (depending only on $W,\bm{v}$ and $A$) such that, 
for $\varepsilon$ sufficiently small, if $\|{\bm{\ut}-\bm{\tilde{v}}}\|_{L^1}\leq\delta$, then
\begin{equation*}\label{lower-sys}
	P_\varepsilon[\bm{u}]\geq P_0[\bm{v}]-C\exp(-A/\varepsilon).
\end{equation*}
\end{prop}

Regarding the initial data, as in \eqref{ass-u0} and \eqref{energy-ini-Neu}, we assume 
\begin{equation}\label{L1}
	\lim_{\varepsilon\rightarrow 0} \|\bm{u}_0^\varepsilon-\bm{v}\|_{{}_{L^1}}=0,
\end{equation} 
and that  at  $t=0$ the energy \eqref{energy-sys} satisfies
\begin{equation}\label{energy0}
	E_\varepsilon[\bm{u}_0^\varepsilon, \bm{\ut}_1^\varepsilon]\leq P_0[\bm{v}]+\frac1{f(\varepsilon)},
\end{equation}
for all $\varepsilon\in(0,\varepsilon_0)$ and 
for some function $f:(0,+\infty)\rightarrow(0,+\infty)$.
Using the equality \eqref{eq:energy-der-sys}, Proposition \ref{prop:lower-sys} and the same arguments of the proof of Theorem \ref{main-scalar-Dir},
we can prove our main result in the case of systems.
\begin{thm}\label{main-sys}
Assume that $W$ satisfies \eqref{W1}-\eqref{W2}. 
Let $\bm{v}$ as defined before and $A\in(0,r\sqrt{2\lambda_0})$. If  $\bm{u}^\varepsilon$ is solution of \eqref{system-CH}-\eqref{Neumann-sys}-\eqref{cond-iniz-sys} 
with initial data $\bm{u}_0^{\varepsilon}$, $\bm{u}_1^{\varepsilon}$ satisfying \eqref{ass-u1-sys}, \eqref{L1} and \eqref{energy0},
then, 
\begin{equation*}\label{limitNeu-sys}
	\sup_{0\leq t\leq  T_\e}\|\bm{\ut}^\varepsilon(\cdot,t)-\bm{\tilde v}\|_{{}_{L^1}}\xrightarrow[\varepsilon\rightarrow0]{}0,
\end{equation*}
where $T_\varepsilon:=\min\{f(\e),\exp(A/\varepsilon)\}$.
Moreover, for any $0<\eta<1$,
\begin{equation*}\label{limit-sys}
	\sup_{0\leq t\leq \e^\eta T_\e}\|\bm{u}^\varepsilon(\cdot,t)-\bm{v}\|_{{}_{L^1}}\xrightarrow[\varepsilon\rightarrow0]{}0.
\end{equation*}
\end{thm}
Thanks to Theorem \ref{main-sys}, we can estimate the velocity of the transition points $h_1,\dots,h_N$ using the standard procedure of Section \ref{sec:layer}.
Let us introduce the definition of \emph{interface} in the case of vector valued function.
If $\bm{v}:[a,b]\to\R^m$ is a step function with jumps at $h_1,h_2,\ldots,h_N$, then its {\it interface} $I[\bm{v}]$ is defined by 
\begin{equation*}
	I[\bm{v}]:=\{h_1,h_2,\ldots,h_N\}.
\end{equation*}
For an arbitrary function $\bm{u}:[a,b]\rightarrow\mathbb{R}^m$ and an arbitrary closed subset $K\subset\R^m\backslash W^{-1}(\{0\})$,
the {\it interface} $I_K[\bm{u}]$ is defined by
\begin{equation*}
	I_K[\bm{u}]:=\bm{u}^{-1}(K).
\end{equation*}

\begin{thm}\label{thm:interface-sys}
Assume that $W$ satisfies \eqref{W1}-\eqref{W2}. 
Let $\bm{u}^\varepsilon$ be solution of \eqref{system-CH}-\eqref{Neumann-sys}-\eqref{cond-iniz-sys} 
with initial data $\bm{u}_0^{\varepsilon}$, $\bm{u}_1^{\varepsilon}$ satisfying \eqref{ass-u1-sys}, \eqref{L1} and \eqref{energy0}.
Given $\delta_1\in(0,r)$ and a closed subset $K\subset\R^m\backslash W^{-1}(\{0\})$, set
\begin{equation*}
	t_\varepsilon(\delta_1)=\inf\{t:\; d(I_K[\bm u^\varepsilon(\cdot,t)],I_K[\bm{u}_0^\varepsilon])>\delta_1\}.
\end{equation*}
There exists $\varepsilon_0>0$ such that if $\varepsilon\in(0,\varepsilon_0)$ then
\begin{equation*}\label{T-interface}
	t_\varepsilon(\delta_1)> \min\{f(\e),\exp(A/\varepsilon)\}.
\end{equation*}
\end{thm}
For the sake of completeness, we report here below a brief description of the arguments needed to prove this result in the present case.
\begin{proof}[Sketch of the proof of Theorem \ref{thm:interface-sys}.]
Fix $\delta_1\in(0,r)$ and a closed subset $K\subset\R^m\backslash W^{-1}(\{0\})$.
In the following, we will denote by $B(\bm{z},\rho)$ the ball of center $\bm{z}$ and of radius $\rho$.
We claim that there exists $\hat\delta>0$ such that for any function $\bm{u}:[a,b]\to\R^m$ satisfying
\begin{equation}\label{eq:tilde}
	\|\bm{\ut}-\bm{\tilde{v}}\|_{{}_{L^1}}<\hat\delta
\end{equation}
and 
\begin{equation}\label{eq:P[u]}
	P_\varepsilon[\bm{u}]\leq P_0[\bm{v}]+ 2N\sup\{\phi(\bm{z}_j,\bm{\xi}) : \bm{z}_j\in W^{-1}(\{0\}), \bm{\xi}\in B(\bm{z}_j,\rho)\},
\end{equation}
for any $\varepsilon\in(0,\varepsilon_0)$, we have
\begin{equation}\label{interfaces-sys}
	d(I_K[\bm{u}], I[\bm{v}])<\tfrac12\delta_1.
\end{equation}
In order to verify the claim, choose $\rho>0$ small enough that
\begin{align*}
	\inf\{ & \phi(\bm{\xi}_1,\bm{\xi}_2) : \bm{z}_j\in W^{-1}(\{0\}), \bm{\xi}_1\in K, \bm{\xi}_2\in B(\bm{z}_j,\rho)\}\\
	&>4N\sup\{\phi(\bm{z}_j,\bm{\xi}_2) : \bm{z}_j\in W^{-1}(\{0\}), \bm{\xi}_2\in B(\bm{z}_j,\rho)\}.
\end{align*}
By using the same arguments of the proof of \eqref{2points} in Proposition \ref{prop:lower}, we can prove that for each $i$ there exist
\begin{equation*}
	x^-_{i}\in(h_i-\delta_1/2,h_i) \qquad \textrm{and} \qquad x^+_{i}\in(\gamma_i,\gamma_i+\delta_1/2)
\end{equation*}
such that
\begin{equation*}
	|\bm{u}(x^-_{i})-\bm{v}(x^-_{i})|<\rho \qquad \textrm{and} \qquad |\bm{u}(x^+_{i})-\bm{v}(x^+_{i})|<\rho.
\end{equation*}
If \eqref{interfaces-sys} is violated, then
\begin{align}
	P_\varepsilon[\bm{u}]&\geq\sum_{i=1}^N \int_{x_i^-}^{x_i^+}\left[\frac\varepsilon2 |\bm{u}_x(x,t)|^2+\frac{W(\bm{u}(x,t))}\varepsilon\right]dx\notag\\
	&+\inf\{\phi(\bm{\xi}_1,\bm{\xi}_2) : \bm{z}_j\in W^{-1}(\{0\}), \bm{\xi}_1\in K, \bm{\xi}_2\in B(\bm{z}_j,\rho)\}. \label{diseq:E1-sys}
\end{align}
On the other hand, triangle inequality gives 
\begin{equation*}
	\phi\bigl(\bm{v}(x^+_{i}),\bm{v}(x^-_{i})\bigr)\leq\phi\bigl(\bm{v}(x^+_{i}),\bm{u}(x^+_{i})\bigr)
	+\phi\bigl(\bm{u}(x^+_{i}),\bm{u}(x^-_{i})\bigr)+\phi\bigl(\bm{u}(x^-_{i}),\bm{v}(x^-_{i})\bigr)
\end{equation*}
and as a consequence
\begin{align*}
	\phi\bigl(\bm{u}(x^-_{i}),\bm{u}(x^+_{i})\bigr)\geq &\, \phi\bigl(\bm{v}(x^+_{i}),\bm{v}(x^-_{i})\bigr)\\
	& -2\sup\{\phi(\bm{z}_j,\bm{\xi}_2) : \bm{z}_j\in W^{-1}(\{0\}), \bm{\xi}_2\in B(\bm{z}_j,\rho)\}. 
\end{align*}
Substituting the latter bound in \eqref{diseq:E1-sys} and using Young's inequality, we infer
\begin{align*}
	P_\varepsilon[\bm{u}]\geq P_0[\bm{v}]&-2N\sup\{\phi(\bm{z}_j,\bm{\xi}_2) : \bm{z}_j\in W^{-1}(\{0\}), \bm{\xi}_2\in B(\bm{z}_j,\rho)\}\\
	&+\inf\{\phi(\bm{\xi}_1,\bm{\xi}_2) : \bm{z}_j\in W^{-1}(\{0\}), \bm{\xi}_1\in K, \bm{\xi}_2\in B(\bm{z}_j,\rho)\}.
\end{align*}
For the choice of $\rho$ and \eqref{eq:P[u]}, we obtain
\begin{align*}
	P_\varepsilon[\bm{u}]>P_0[\bm{v}]+2N\sup\{\phi(\bm{z}_j,\bm{\xi}_2) : \bm{z}_j\in W^{-1}(\{0\}), \bm{\xi}_2\in B(\bm{z}_j,\rho)\}
	\geq P_\varepsilon[\bm{u}],
\end{align*}
which is a contradiction. 
Hence, the claim \eqref{interfaces-sys} is true.

Now, we conclude the proof of Theorem \ref{thm:interface-sys} by applying \eqref{interfaces-sys} to the initial datum $\bm{u}_0^\e$
and to the solution $\bm{u}^\e(\cdot,t)$ for any $t\in(0,\min\{f(\e),\exp(A/\varepsilon)\})$.
The assumptions on the initial data \eqref{L1}, \eqref{energy0} imply that $\bm{u}_0^\varepsilon$ satisfy \eqref{eq:tilde} and \eqref{eq:P[u]} 
if $\e$ is sufficiently small.
Then, we have
\begin{equation}\label{interfaces-u0-sys}
	d(I_K[\bm{u}_0^\varepsilon], I[\bm{v}])<\tfrac12\delta_1.
\end{equation}
From Theorem \ref{main-sys}, it follows that the solution $\bm{u}^\varepsilon(\cdot,t)$ satisfies 
the condition \eqref{eq:tilde} for any $t\in(0,\min\{f(\e),\exp(A/\varepsilon)\})$,
while \eqref{eq:P[u]} holds because $E_\varepsilon[\bm{u}^\varepsilon,\bm{\ut}^\varepsilon_t](t)$ is a nonincreasing function of $t$
along the solutions of \eqref{system-CH} with boundary conditions \eqref{Neumann-sys} and initial datum $\bm{u}_1$ satisfying \eqref{ass-u1-sys}.
Therefore,
\begin{equation}\label{interfaces-u-sys}
	d(I_K[\bm{u}^\varepsilon(t)], I[\bm{v}])<\tfrac12\delta_1,
\end{equation}
for any $t\in(0,\min\{f(\e),\exp(A/\varepsilon)\})$.
Combining \eqref{interfaces-u0-sys} and \eqref{interfaces-u-sys}, we obtain
\begin{equation*}
	d(I_K[\bm{u}^\varepsilon(t)],I_K[\bm{u}_0^\varepsilon])<\delta_1
\end{equation*}
for any $t\in(0,\min\{f(\e),\exp(A/\varepsilon)\})$ and the proof is complete.
\end{proof}

To conclude our studies, we point out that the construction of an initial profile with $N$--transitions layer structure of Section \ref{sec:N-translayer} can be extended in the present vectorial case without significant changes; for details, we refer to \cite{Grant,Folino2}.

\section*{\bf Conflict of interest statement}
The authors have no conflict of interest to declare.

\end{document}